\documentclass[reqno,11pt]{amsart}
\usepackage[foot]{amsaddr}
\usepackage{bbold}
\usepackage{charter}
\usepackage[margin=1in]{geometry}
\usepackage[colorlinks=true,linktoc=all,linkcolor=blue,citecolor=blue]{hyperref}

\theoremstyle{plain}
\newtheorem{theorem}{Theorem}[section]
\newtheorem*{theorem*}{Theorem}
\newtheorem{proposition}{Proposition}[section]
\newtheorem{lemma}[theorem]{Lemma}
\newtheorem{corollary}[theorem]{Corollary}
\theoremstyle{definition}

\newtheorem{remark}[theorem]{Remark}
\newtheorem{example}[theorem]{Example}

\numberwithin{equation}{section}

\renewcommand{\hat}{\widehat}
\newcommand{\bigchi}{\mbox{\Large$\chi$}}
\newcommand{\metric}{\mathrm{d}}
\newcommand{\Co}{\mathbb{C}}
\newcommand{\D}{\mathbb{D}}
\newcommand{\N}{\mathbb{N}}
\newcommand{\Zed}{\mathbb{Z}}

\newcommand{\Lmuinf}{L_{\hspace{-.2ex}\mu}^{\infty}}
\newcommand{\lilLmuinf}{L_{\hspace{-.2ex}\mu}^0}
\newcommand{\adWCO}{W_{\hspace{-.2ex}\psi,\varphi}^*}
\newcommand{\modu}[1]{\left\lvert #1 \right\rvert}

\title{Weighted Composition Operators on Discrete Weighted Banach Spaces}

\author{%
	Robert F.~Allen\textsuperscript{1} and 
	Matthew A.~Pons\textsuperscript{2}}
\address{\textsuperscript{1}Department of Mathematics \& Statistics, University of Wisconsin-La Crosse}
\address{\textsuperscript{2}Department of Mathematics \& Actuarial Science, North Central College}
\email{rallen@uwlax.edu, mapons@noctrl.edu}


\begin{document}
	
	\begin{abstract} We present the current results in the study of weighted composition operators on weighted Banach spaces of an unbounded, locally finite metric space.  Specifically, we determine characterizations of bounded and compact weighted composition operators, including the operator and essential norms.  In addition, we characterize the weighted composition operators that are injective, are bounded below, have closed range, and have bounded inverse.  We characterize the isometries and surjective isometries among the weighted composition operators, as well as those that satisfy the Fredholm condition.  Lastly, we provide numerous interesting examples of the richness of these operators acting on the discrete weighted Banach spaces. 
	\end{abstract}
	
	\maketitle
	
	\section{Introduction}
	
	Let $X$ be a Banach space of functions on a domain $\Omega$.  For $\psi$ a function on $\Omega$ and $\varphi$ a self-map of $\Omega$, the linear operator defined on $X$ by \[W_{\psi,\varphi}f = \psi(f\circ\varphi)\] is called the weighted composition operator induced by $\psi$ and $\varphi$.  
	Observe that when $\psi \equiv 1$, we have the composition operator $C_\varphi f = f\circ \varphi$, and similarly when $\varphi(z) = z$, we have the multiplication operator $M_\psi f = \psi f.$
	
	Classically, the study of weighted composition operators has been linked to isometries on Banach spaces.  In fact, Banach \cite{Banach:32} proved that the surjective isometries on $C(Q)$, the space of continuous real-valued functions on a compact metric space $Q$, are of the form $f \mapsto \psi(f\circ\varphi)$, where $\vert\psi\vert \equiv 1$ and $\varphi$ is a homeomorphism of $Q$ onto itself.  The characterization of isometries on most Banach spaces of analytic functions is still an open problem.  However, there are many spaces for which the isometries are known.  In many of these cases, the isometries have the form of a weighted composition operator.  The interested reader is directed to \cite{CimaWogen:80,ElGebeilyWolfe:85,Kolaski:82}.
	
	The study of weighted composition operators is not limited to the study of isometries.  Moreover, properties of weighted composition operators are not solely determined by the composition and multiplication operators of which they are comprised.  There are many examples of bounded (compact) weighted composition operators that are not comprised of bounded (compact) composition or multiplication operators.  In the last section of this paper, we provide further examples of such weighted composition operators.
	
	In recent years, spaces of functions defined on discrete structures such as infinite trees have been explored.  These spaces provide discrete analogs to classical spaces of analytic functions on the open unit disk $\D$ in $\Co$.  A discrete version of the Bloch space was developed by Colonna and Easley \cite{ColonnaEasley:10} called the Lipschitz space.  Further research on the Lipschitz space, as well as multiplication and composition operators acting on the Lipschitz space, has been conducted by Colonna, Easley, and the first author.  The interested reader is directed to \cite{ColonnaEasley:10,AllenColonnaEasley:14}.  In addition, a discrete analog of the Hardy space was developed by Muthukumar and Ponnusamy, and the multiplication and composition operators were studied \cite{MuthukumarPonnusamyI,MuthukumarPonnusamyII}.
	
	Of interest in this paper is a discrete analog to the weighted Banach spaces $H_\nu^\infty(\D)$, and their weighted composition operators, as studied in \cite{BonetLindstromWolf:2008,ContrerasHernandez-Diaz:2000,Montes:2000}.  In \cite{AllenCraig:2015}, the discrete weighted Banach space $\Lmuinf$ was defined, and the multiplication operators were studied.  In \cite{AllenPons:2016}, the authors furthered the operator theory on $\Lmuinf$ by studying the composition operators.  The study of composition operators on such discrete spaces poses more challenges than the study of multiplication operators.  In this paper, we study the weighted composition operators on $\Lmuinf$.  
	
	\subsection{Organization of the paper} In Section \ref{Section:Space}, we collect useful results on weighted Banach spaces of an unbounded, locally finite metric space, as well as the little weighted Banach space.
	
	In Section \ref{Section:Bounded}, we characterize the bounded weighted composition operators as well as determine their operator norms.  We also provide necessary and sufficient conditions for the weighted composition operator to be bounded on the little weighted Banach space, while providing a complete characterization in two situations. 
	
	In Section \ref{Section:Compactness}, we characterize the compact weighted composition operators and determine their essential norm.  These results lead to the characterization of compact multiplication and composition operators on the little weighted Banach space, which have not previously been studied.
	
	In Section \ref{Section:BoundedBelow}, we characterize the weighted composition operators that are injective, are bounded below, and have closed range.  The application of these results to the multiplication operator yields the characterization of bounded below as in \cite{AllenCraig:2015}, but with a completely different proof.  
	
	In Section \ref{Section:Invertible}, we characterize the weighted composition operators that are invertible with bounded inverse.  In addition, we characterize the isometries and surjective isometries among the weighted composition operators.  This completes the characterization of the isometries and surjective isometries amongst the composition operators that was started in \cite{AllenPons:2016}.
	
	In Section \ref{Section:Fredholm}, we characterize the so-called Fredholm weighted composition operators.  This gives rise to characterizations of the Fredholm multiplication and composition operators as well.  To date, this is the first study of Fredholm operators on such discrete spaces.
	
	Finally, in Section \ref{Section:Examples}, we illustrate the richness of the weighted composition operators acting on the weighted Banach spaces through several examples.  We show in many cases that the weighted composition operator is more than the sum of its parts.  Among the examples is a compact weighted composition operator for which neither the corresponding multiplication or composition operators are compact, and an isometric weighted composition operator for which the composition operator is not bounded.
	
	\subsection{Preliminary definitions and notation}
	The domains of the function spaces in this paper are metric spaces that are locally finite, with a distinguished element $o$, called the root.  Recall, a metric space $(T,\metric)$ is locally finite if for every $M>0$, the set $\{v \in T : \metric(o,v) < M\}$ is finite.  For a point $v$ in $T$, we define the length of $v$ by $\modu{v} = \metric(o,v)$.  In this paper, we assume the locally finite metric space $(T,\metric)$ has root $o$ and is unbounded, that is for every $M > 0$, there exists $v \in T$ with $\modu{v} \geq M$.  As the length of a point is used throughout, and not specifically the metric $\metric$, we will denote the metric space simply by $T$. Lastly, we denote by $T^*$ the set $T\setminus\{o\}$.  
	
	\section{Weighted Banach Spaces}\label{Section:Space}
	In this section, we define the weighted Banach spaces of an unbounded, locally finite metric space $T$, and collect useful results for this paper.  A positive function $\mu$ on $T$ is called a \textit{weight}.  The \textit{weighted Banach space} on $T$ with weight $\mu$, denoted $\Lmuinf(T)$ or simply $\Lmuinf$, is defined as the space of functions $f$ on $T$ for which \[\sup_{v \in T}\;\mu(v)\modu{f(v)} < \infty.\]
	
	\noindent The \textit{little weighted Banach space} on $T$ with weight $\mu$, denoted $\lilLmuinf(T)$ or simply $\lilLmuinf$, is the space of functions $f \in \Lmuinf$ for which \[\lim_{\modu{v} \to \infty} \mu(v)\modu{f(v)} = 0.\] It was shown in \cite{AllenCraig:2015} that, when $T$ is an infinite rooted tree, the space $\Lmuinf(T)$ endowed with the norm \[\|f\|_\mu = \sup_{v \in T} \mu(v)\modu{f(v)}\] is a functional Banach space, that is, a Banach space for which every point-evaluation functional $K_v:\Lmuinf(T) \to \Co$, $f\mapsto f(v)$, is bounded for all $v \in T$.  The proof of \cite{AllenCraig:2015} carries forward for a locally finite metric space $T$.
	It was shown in \cite{AllenColonnaMartinesPons} that $\lilLmuinf$ is a closed, separable subspace of $\Lmuinf$.   The following lemmas capture the properties most relevant to our work here. We note that similar statements can be made for $\lilLmuinf$.
	
	\begin{lemma}[{\cite[Lemma 2.6]{AllenPons:2016}}]\label{Lemma:point_evaluation_bound}
		Suppose $f$ is a function in $\Lmuinf$.  Then for all $v \in T$, it holds that \[\modu{f(v)} \leq \frac{1}{\mu(v)}\|f\|_\mu.\]
	\end{lemma}
	
	We call a weight function $\mu$ \textit{typical} if $\lim_{\modu{v}\to\infty} \mu(v) = 0.$  The next result shows that little weighted Banach spaces containing the constant functions are precisely those with a typical weight.
	
	\begin{lemma}\label{Lemma:ConstantsTypicalWeight} The constant function 1 is an element of $\lilLmuinf$ if and only if $\mu$ is a typical weight.
	\end{lemma}
	
	\begin{lemma}\label{Lemma:1/mu_in_lilLmuinf}
		For $w \in T$, the functions $f(v) = \bigchi_w(v)$ and $g(v) = \frac{1}{\mu(v)}\bigchi_{w}(v)$ are elements of $\lilLmuinf$ with $\|f\|_\mu = \mu(w)$ and $\|g\|_\mu = 1$.
	\end{lemma}
	
	\begin{lemma}\label{lemma:evaluationfunctionalsindep}
		If $\{v_i\}_{i=1}^n$ is a set of distinct points in $T$, then the set of point-evaluation functionals $\{K_{v_i}\}_{i=1}^n$ is linearly independent in $(\Lmuinf)^*$.
	\end{lemma}
	
	\begin{proof}
		The statement follows immediately by considering functions $f_i(v)=\bigchi_{v_i}(v)$ since each $f_i$ vanishes everywhere except $v_i$.
	\end{proof}
	
	\section{Boundedness and Operator Norm}\label{Section:Bounded}
	In this section, we study the boundedness of weighted composition operators acting on $\Lmuinf$ and $\lilLmuinf$.  In this endeavor, we define the following quantities for $\psi$ a function on $T$ and $\varphi$ a self-map of $T$:
	\[\sigma_{\psi,\varphi} = \sup_{v \in T}\;\frac{\mu(v)}{\mu(\varphi(v))}\modu{\psi(v)}\] and \[\xi_{\psi,\varphi} = \lim_{\modu{v} \to \infty} \frac{\mu(v)}{\mu(\varphi(v))}\modu{\psi(v)},\] if the limit exists.  For the boundedness of $W_{\psi,\varphi}$ on $\Lmuinf$, the quantity $\sigma_{\psi,\varphi}$ is the characterizing quantity.
	
	\begin{remark}\label{Remark:xi_equivalence}
		Note if $\varphi$ is a self-map of $T$ with finite range, them  $\xi_{\psi,\varphi} = 0$ if and only if $\lim_{\modu{v}\to\infty} \mu(v)\modu{\psi(v)} = 0$.  This follows directly from the definition of $\xi_{\psi,\varphi}$ and the existence of positive constants $m,M$ such that $m \leq \mu(\varphi(v)) \leq M$ for all $v \in T$. 
	\end{remark}
	
	We summarize the main results of this section in the following theorem.
	
	\begin{theorem*}\label{Theorem:Section3_summary}
		Let $\psi$ be a function on $T$ and $\varphi$ a self-map of $T$.
		\begin{enumerate}
			\item[(a)] The operator $W_{\psi,\varphi}:\Lmuinf\to\Lmuinf$ is bounded if and only if $\sigma_{\psi,\varphi}$ is finite.  In this case, $\|W_{\psi,\varphi}\| = \sigma_{\psi,\varphi}$.
			\item[(b)] For the operator $W_{\psi,\varphi}:\lilLmuinf\to\lilLmuinf$, 
			\begin{enumerate}
				\item[i.] if $\varphi$ has finite range, then $W_{\psi,\varphi}$ is bounded if and only if $\xi_{\psi,\varphi} = 0$.  In this case, $\|W_{\psi,\varphi}\| = \sigma_{\psi,\varphi}$.
				\item[ii.] if $\varphi$ has infinite range and $\mu$ is a typical weight, then $W_{\psi,\varphi}$ is bounded if and only if $\psi \in \lilLmuinf$ and $\sigma_{\psi,\varphi}$ is finite.
			\end{enumerate}
		\end{enumerate}
	\end{theorem*}
	
	\noindent In the remainder of the section, we provide proofs to the elements of the above theorem, along with useful lemmas in a more digestible format.
	
	\begin{theorem}\label{Theorem:bounded_criteria_Lmuinfty}
		Suppose $\psi$ is a function on $T$ and $\varphi$ is a self-map of $T$.  Then $W_{\psi,\varphi}$ is bounded on $\Lmuinf$ if and only if  $\sigma_{\psi,\varphi}$ is finite.  Moreover, it holds that \[\|W_{\psi,\varphi}\| = \sigma_{\psi,\varphi}.\]
	\end{theorem}
	
	\begin{proof}
		Suppose $W_{\psi,\varphi}$ is a bounded operator on $\Lmuinf$.  We define the function $g(v) = \frac{1}{\mu(v)}$, which is an element of $\Lmuinf$ with $\|g\|_\mu = 1$.  For a fixed point $w \in T$,  it holds that
		\begin{equation}
			\begin{aligned}
				\frac{\mu(w)}{\mu(\varphi(w))}\modu{\psi(w)} &\leq \sup_{v \in T}\;\frac{\mu(v)}{\mu(\varphi(v))}\modu{\psi(v)}\\
				&= \sup_{v \in T}\;\mu(v)\modu{\psi(v)}\modu{g(\varphi(v))}\\
				&= \|W_{\psi,\varphi} g\|_\mu\\
				&\leq \|W_{\psi,\varphi}\|.
		\end{aligned}\end{equation}  Taking the supremum over all $w \in T$, it follows that $\sigma_{\psi,\varphi} \leq \|W_{\psi,\varphi}\|.$  Thus $\sigma_{\psi,\varphi}$ is finite.
		
		Conversely, suppose $\sigma_{\psi,\varphi}$ is finite and let $f \in \Lmuinf$ with $\|f\|_\mu \leq 1$.  From Lemma \ref{Lemma:point_evaluation_bound}, it follows that
		\begin{equation}\label{Inequality:norm_upper_bound}
			\|W_{\psi,\varphi}f\|_\mu = \sup_{v \in T}\;\mu(v)\modu{\psi(v)}\modu{f(\varphi(v))} \leq \sup_{v \in T}\;\frac{\mu(v)}{\mu(\varphi(v))}\modu{\psi(v)}\|f\|_\mu \leq \sigma_{\psi,\varphi}.
		\end{equation}  Thus, $W_{\psi,\varphi}$ is a bounded operator on $\Lmuinf$.  Taking the supremum over all such functions $f$, we obtain $\|W_{\psi,\varphi}\| \leq \sigma_{\psi,\varphi}$.
	\end{proof}
	
	In much of the analysis for weighted composition operators $W_{\psi,\varphi}$ on the discrete weighted Banach spaces, the behavior of the operator depends on the image of $T$ under $\varphi$.  We study the behavior in terms of $\varphi$ having either finite or infinite range.  
	When a self-map $\varphi$ of $T$ has \textit{infinite range} then, since $T$ is locally finite, there must exist a sequence of points $(v_n)$ in $T$ with $\modu{v_n} \to \infty$ such that $\modu{\varphi(v_n)} \to \infty$.
	
	In the rest of this section, we characterize the boundedness of $W_{\psi,\varphi}$ on $\lilLmuinf$. By the Closed Graph Theorem and the boundedness of the evaluation functionals, to show the weighted composition operator is bounded, it suffices to show it maps $\lilLmuinf$ into itself.  We will exploit this reduction frequently.  In the next two results, we show that $\xi_{\psi,\varphi} = 0$ is a sufficient condition for boundedness on $\lilLmuinf$. and, for $\varphi$ with finite range, it is also necessary.
	
	\begin{lemma}\label{Lemma:boundedness_half_little_space}
		Suppose $\psi$ is a function on $T$ and $\varphi$ is a self-map of $T$.  If $\xi_{\psi,\varphi} = 0$, then $W_{\psi,\varphi}$ is bounded on $\lilLmuinf$.  Moreover, it holds that $\|W_{\psi,\varphi}\| \leq \sigma_{\psi,\varphi}$.
	\end{lemma}
	
	\begin{proof}
		Suppose $\xi_{\psi,\varphi} = 0$.  It follows that $\sigma_{\psi,\varphi}$ is finite, and thus $W_{\psi,\varphi}$ is bounded as an operator on $\Lmuinf$.  To show $W_{\psi,\varphi}$ is bounded on $\lilLmuinf$, it suffices to show $W_{\psi,\varphi}$ maps $\lilLmuinf$ into $\lilLmuinf$.  Let $f \in \lilLmuinf$ and $(v_n)$ be a sequence in $T$ with $\modu{v_n} \to \infty$ as $n \to \infty$.  From Lemma \ref{Lemma:point_evaluation_bound} it follows that
		\[\mu(v_n)\modu{\psi(v_n)}\modu{f(\varphi(v_n))} \leq \frac{\mu(v_n)}{\mu(\varphi(v_n))}\modu{\psi(v_n)}\|f\|_\mu \to 0\] as $n \to \infty$.  Thus $W_{\psi,\varphi}$ is bounded on $\lilLmuinf$.  Moreover, from  \eqref{Inequality:norm_upper_bound} it follows that $\|W_{\psi,\varphi}\| \leq \sigma_{\psi,\varphi}$.
	\end{proof}
	
	\begin{theorem}\label{Theorem:bounded_criteria_Lmu0_elliptic}
		Suppose $\psi$ is a function on $T$ and $\varphi$ is a self-map of $T$ with finite range.  Then $W_{\psi,\varphi}$ is bounded on $\lilLmuinf$ if and only if $\xi_{\psi,\varphi} = 0.$  Moreover, it holds that \[\|W_{\psi,\varphi}\| = \sigma_{\psi,\varphi}.\]
	\end{theorem}
	
	\begin{proof}
		Suppose $W_{\psi,\varphi}$ is bounded on $\lilLmuinf$.  Since $\varphi$ has finite range, it follows that the function $g(v) = \frac{1}{\mu(v)}\bigchi_{\varphi(T)}(v)$ is in $\lilLmuinf$ with $\|g\|_\mu = 1$.  Thus $W_{\psi,\varphi} g$ is in $\lilLmuinf$ as well.  For $v \in T$, we have
		\begin{equation}\label{Equality:elliptic_bounded}
			\begin{aligned}
				\frac{\mu(v)}{\mu(\varphi(v))}\modu{\psi(v)} &= \frac{\mu(v)}{\mu(\varphi(v))}\modu{\psi(v)}\bigchi_{\varphi(T)}(\varphi(v))\\
				&= \mu(v)\modu{\psi(v)}\modu{g(\varphi(v))}\\
				&= \mu(v)\modu{(W_{\psi,\varphi} g)(v)}.
			\end{aligned}
		\end{equation}  It immediately follows that $\xi_{\psi,\varphi} = 0$.  Moreover, it holds that $\sigma_{\psi,\varphi} \leq \|W_{\psi,\varphi}\|$.  The converse follows from Lemma \ref{Lemma:boundedness_half_little_space}.
	\end{proof}
	
	The following lemma shows that boundedness on $\lilLmuinf$ implies boundedness on $\Lmuinf$.  With this result, we characterize the boundedness of $W_{\psi,\varphi}$ on $\lilLmuinf$ under a typical weight.  As will be shown in future sections, the inverse image of $\varphi(w) \in T$ under $\varphi$ will play a role in determining characteristics of the weighted composition operator.  To this end, for a point $w \in T$ and $\varphi$ a self-map of $T$, we define $S_w = \varphi^{-1}(\varphi(w))$.
	
	\begin{lemma}\label{Lemma:Bounded_inf_0}
		Suppose $\psi$ is a function on $T$ and $\varphi$ is a self-map of $T$.  If $W_{\psi,\varphi}$ is bounded on $\lilLmuinf$, then $W_{\psi,\varphi}$ is bounded on $\Lmuinf$.
	\end{lemma}
	
	\begin{proof}
		By Theorem \ref{Theorem:bounded_criteria_Lmuinfty}, it suffices to show that $\sigma_{\psi,\varphi} < \infty$.  Fix $w \in T$ and define the function $g(v) = \frac{1}{\mu(v)}\bigchi_{\varphi(w)}(v)$.  From Lemma \ref{Lemma:1/mu_in_lilLmuinf}, $g \in \lilLmuinf$ with $\|g\|_\mu = 1$.  Define $Y = \{f \in \lilLmuinf : \|f\|_\mu = 1\}$.  Then
		\begin{align*}
			\frac{\mu(w)}{\mu(\varphi(w))}\modu{\psi(w)} &\leq \sup_{v \in S_w}\;\frac{\mu(v)}{\mu(\varphi(v))}\modu{\psi(v)}\\
			&= \sup_{v \in T}\;\frac{\mu(v)}{\mu(\varphi(v))}\modu{\psi(v)}\bigchi_{\varphi(w)}(\varphi(v))\\
			&= \|W_{\psi,\varphi} g\|_\mu\\
			&\leq \sup_{f \in Y}\;\|W_{\psi,\varphi} f\|_\mu\\
			&= \left\|W_{\psi,\varphi}:\lilLmuinf\to\lilLmuinf\right\|.
		\end{align*} Taking the supremum over all $w \in T$, we obtain $\sigma_{\psi,\varphi} \leq \left\|W_{\psi,\varphi}:\lilLmuinf \to \lilLmuinf\right\|$.
		Since $W_{\psi,\varphi}$ is bounded on $\lilLmuinf$, then $\sigma_{\psi,\varphi} < \infty$ and hence $W_{\psi,\varphi}$ is bounded on $\Lmuinf$ by Theorem \ref{Theorem:bounded_criteria_Lmuinfty}.
	\end{proof}
	
	\begin{theorem}
		Let $\mu$ be a typical weight.  Suppose $\psi$ is a function on $T$ and $\varphi$ is a self-map of $T$ with infinite range.  Then $W_{\psi,\varphi}$ is bounded on $\lilLmuinf$ if and only if $\psi \in \lilLmuinf$ and $\sigma_{\psi,\varphi} < \infty$.
	\end{theorem}
	
	\begin{proof}
		First suppose $W_{\psi,\varphi}$ is bounded on $\lilLmuinf$.  Since $\mu$ is a typical weight, the constant function $1$ is an element of $\lilLmuinf$ from Lemma \ref{Lemma:ConstantsTypicalWeight}.  From the boundedness of $W_{\psi,\varphi}$, we have that $W_{\psi,\varphi} 1 = \psi$ is also an element of $\lilLmuinf$.  In addition, Lemma \ref{Lemma:Bounded_inf_0} implies $W_{\psi,\varphi}$ is bounded as an operator on $\Lmuinf$, and thus $\sigma_{\psi,\varphi} < \infty$ by Theorem \ref{Theorem:bounded_criteria_Lmuinfty}.
		
		Next, suppose $\psi \in \lilLmuinf$ and $\sigma_{\psi,\varphi} < \infty$.  To prove the boundedness of $W_{\psi,\varphi}$, it suffices to show the operator maps $\lilLmuinf$ into $\lilLmuinf$.  Let $\varepsilon > 0$ and $f \in \lilLmuinf$.  There exists a natural number $N_1$ such that if $\modu{v} > N_1$, then $\mu(v)\modu{f(v)} < \frac{\varepsilon}{\sigma_{\psi,\varphi}}$.  Define $m = 1+\sup_{\modu{v} \leq N_1}\;\modu{f(v)}$ and observe this quantity is finite and non-zero since the set $\{v \in T : \modu{v} \leq N_1\}$ is finite.  In addition, there exists a natural number $N_2$ such that if $\modu{v} > N_2$, then $\mu(v)\modu{\psi(v)} < \frac{\varepsilon}{m}$.
		
		Let $v \in T$ such that $\modu{v} > N_2$.  If $\modu{\varphi(v)} > N_1$, then
		\begin{align*}\mu(v)\modu{\psi(v)}\modu{f(\varphi(v))} &= \frac{\mu(v)}{\mu(\varphi(v))}\modu{\psi(v)}\mu(\varphi(v))\modu{f(\varphi(v))}\\
			&\leq \sigma_{\psi,\varphi}\mu(\varphi(v))\modu{f(\varphi(v))}\\&< \varepsilon.\end{align*}  On the other hand, if $\modu{\varphi(v)} \leq N_1$, then \[\mu(v)\modu{\psi(v)}\modu{f(\varphi(v))} \leq \mu(v)\modu{\psi(v)}\sup_{\modu{v}\leq N_1}\;\modu{f(w)} < \mu(v)\modu{\psi(v)}m < \varepsilon.\]  Thus \[\lim_{\modu{v} \to \infty} \mu(v)\modu{\psi(v)}\modu{f(\varphi(v))} = 0\] and $W_{\psi,\varphi} f \in \lilLmuinf$.
	\end{proof}
	
	We complete this section with boundedness characteristics for composition operators $C_\varphi$ and multiplication operators $M_\psi$ on $\lilLmuinf$, which were not studied in \cite{AllenPons:2016} or \cite{AllenCraig:2015}.  However, bounded composition operators on $\lilLmuinf$ are further studied in \cite{AllenColonnaMartinesPons}.  For the composition operator induced by a self-map $\varphi$ with finite range, the characterization for boundedness from Theorem \ref{Theorem:bounded_criteria_Lmu0_elliptic} translates to $\lim_{\modu{v}\to\infty} \frac{\mu(v)}{\mu(\varphi(v))} = 0$.  From Remark \ref{Remark:xi_equivalence}, this is equivalent to $\lim_{\modu{v}\to\infty} \mu(v) = 0$, i.e., $\mu$ being a typical weight.
	
	\begin{corollary}
		Let $\psi$ be a function on $T$ and $\varphi$ a self-map of $T$.
		\begin{enumerate}
			\item[(a)] For the operator $C_\varphi:\lilLmuinf\to\lilLmuinf$,
			\begin{enumerate}
				\item[i.] If $\varphi$ has finite range, then $C_\varphi$ is bounded on $\lilLmuinf$ if and only if $\mu$ is a typical weight.
				\item[ii.] If $\varphi$ has infinite range and $\mu$ is a typical weight, then $C_\varphi$ is bounded on $\lilLmuinf$ if and only if $C_\varphi$ is bounded on $\Lmuinf$.
			\end{enumerate}
			\item[(b)] If $\mu$ is a typical weight, then $M_\psi:\lilLmuinf \to \lilLmuinf$ is bounded if and only if $\psi \in \lilLmuinf$.
		\end{enumerate}
	\end{corollary}
	
	\section{Compactness and Essential Norm}\label{Section:Compactness}
	In this section, we study the compactness of weighted composition operators on the discrete weighted Banach spaces.  As with boundedness, conditions for compactness depend on the image of $T$ under $\varphi$.  We summarize the main results of this section in the following theorem.
	
	\begin{theorem*}
		Suppose $\psi$ is a function on $T$ and $\varphi$ is a self-map of $T$ for which $W_{\psi,\varphi}$ is bounded on $\Lmuinf$ (respectively $\lilLmuinf$).
		\begin{enumerate}
			\item[(a)] If $\varphi$ has finite range, then $W_{\psi,\varphi}$ is compact on $\Lmuinf$ (respectively $\lilLmuinf$).
			\item[(b)] If $\varphi$ is infinite range, then 
			\begin{enumerate}
				\item[i.] The operator $W_{\psi,\varphi}$ is compact on $\Lmuinf$ (respectively $\lilLmuinf$) if and only if \[\lim_{N \to \infty} \sup_{\modu{\varphi(v)}\geq N}\;\frac{\mu(v)}{\mu(\varphi(v))}\modu{\psi(v)} =0.\]
				\item[ii.] If $\mu$ is a typical weight, then $W_{\psi,\varphi}$ is compact on $\lilLmuinf$ if and only if $\xi_{\psi,\varphi} = 0$.
			\end{enumerate}
		\end{enumerate}
	\end{theorem*}
	
	Our first result in this section shows that self-maps with finite range induce compact weighted composition operators on both $\Lmuinf$ and $\lilLmuinf$, independent of the multiplication symbol.  This result utilizes the sequence characterization of compactness contained in the next lemma.
	
	\begin{lemma}[{\cite[Lemma 2.5]{AllenCraig:2015}}]\label{Lemma:Compactness-criterion}
		Let $X,Y$ be two Banach spaces of functions on an unbounded, locally finite metric space $(T,\metric)$.  Suppose that
		\begin{enumerate}
			\item[(a)] the point evaluation functionals of $X$ are bounded,
			\item[(b)] the closed unit ball of $X$ is a compact subset of $X$ in the topology of uniform convergence on compact sets,
			\item[(c)] $A:X \to Y$ is bounded when $X$ and $Y$ are given the topology of uniform convergence on compact sets.
		\end{enumerate}  Then $A$ is a compact operator if and only if given a bounded sequence $(f_n)$ in $X$ such that $f_n \to 0$ pointwise, then the sequence $(A f_n)$ converges to zero in the norm of $Y$.
	\end{lemma}
	
	\begin{theorem}
		Suppose $\psi$ is a function on $T$ and $\varphi$ is a self-map of $T$ with finite range for which $W_{\psi,\varphi}$ is bounded on $\Lmuinf$ (respectively $\lilLmuinf$).  Then $W_{\psi,\varphi}$ is compact on $\Lmuinf$ (respectively $\lilLmuinf$).
	\end{theorem}
	
	\begin{proof}
		We will prove compactness on $\Lmuinf$, as the proof for the $\lilLmuinf$ case is identical.  Since $W_{\psi,\varphi}$ is bounded, from Theorem \ref{Theorem:bounded_criteria_Lmuinfty} we have that $\sigma_{\psi,\varphi}$ is finite.  Let $(f_n)$ be a bounded sequence in $\Lmuinf$ converging to 0 pointwise and fix $\varepsilon > 0$.  Since $\varphi(T)$ is finite, there exists a positive constant $m$ such that $\sup_{w \in \varphi(T)}\;\mu(w)  \leq m$.  Also, the pointwise convergence of $(f_n)$ to 0 is uniform on $\varphi(T)$.  Thus, for sufficiently large $n$, we have $\sup_{w \in \varphi(T)}\;\modu{f_n(w)} < \frac{\varepsilon}{m \sigma_{\psi,\varphi}}$.  With these observations, we see for such $n$,
		\begin{align*}
			\|W_{\psi,\varphi} f_n\|_\mu &= \sup_{v \in T}\;\mu(v)\modu{\psi(v)f_n(\varphi(v))}\\
			&= \sup_{v \in T}\;\frac{\mu(v)}{\mu(\varphi(v))}\modu{\psi(v)} \mu(\varphi(v)) \modu{f_n(\varphi(v))}\\
			&\leq \sigma_{\psi,\varphi}\sup_{w \in \varphi(T)}\;\mu(w) \modu{f_n(w)}\\
			&\leq m\sigma_{\psi,\varphi} \sup_{w \in \varphi(T)}\;\modu{f_n(w)}\\
			&< \varepsilon.
		\end{align*}  So $\|W_{\psi,\varphi} f_n\|_\mu \to 0$ as $n \to \infty$.  Thus by Lemma \ref{Lemma:Compactness-criterion}, $W_{\psi,\varphi}$ is compact on $\Lmuinf$.
	\end{proof}
	
	In view of the previous theorem, we assume in the rest of this section that $\varphi$ has an infinite range and determine the compactness of the operator $W_{\psi,\varphi}$ by computing its essential norm. To this end, we employ the following sequence of compact operators.  First, for $f\in\Lmuinf$ and $n\in\N$, define a function $f_n\in\Lmuinf$ by \[f_n(v)= \begin{cases}f(v) & \text{if $\modu{v} \leq n$}\\0 & \text{if $\modu{v} > n$}.\end{cases}\]  Then define the operator $A_n$ by $A_n f=f_n$.  It is easy to see that these operators are linear. The following lemma captures the other most relevant properties.
	
	\begin{lemma}\label{Lemma:compactsequence}
		For each $n\in\N$, the operator $A_n$ is compact on $\Lmuinf$ (respectively $\lilLmuinf$) with $\|A_n\|\leq 1$ and $\|I-A_n\|\leq 1.$
	\end{lemma}
	
	\begin{theorem}\label{Theorem:essential_norm_non-typical}
		Suppose $\psi$ is a function on $T$ and $\varphi$ is a self-map of $T$ with infinite range for which $W_{\psi,\varphi}$ is bounded on $\Lmuinf$ (respectively $\lilLmuinf$).  Then \begin{equation}\label{Equation:essentialnorm}\|W_{\psi,\varphi}\|_e = \lim_{N \to \infty} \sup_{\modu{\varphi(v)}\geq N}\;\frac{\mu(v)}{\mu(\varphi(v))}\modu{\psi(v)}\end{equation} as an operator on $\Lmuinf$ (respectively $\lilLmuinf$).
	\end{theorem}
	
	\begin{proof}
		We will compute the essential norm for $W_{\psi,\varphi}$ acting on $\Lmuinf$, as the proof for the $\lilLmuinf$ case is identical.  Observe $W_{\psi,\varphi} A_n$ is compact for all $n \in \N$ since $W_{\psi,\varphi}$ is bounded and $A_n$ is compact from Lemma \ref{Lemma:compactsequence}.  From the definition of the essential norm, we have \begin{equation}
			\label{Inequality:essentialnormupper}\|W_{\psi,\varphi}\|_e\leq \|W_{\psi,\varphi}-W_{\psi,\varphi} A_n\|=\sup_{\|f\|_\mu\leq1}\;\sup_{v\in T}\;\mu(v)\modu{(W_{\psi,\varphi}(I-A_n)f)(v)}\end{equation} for every $n\in\N$.  Now fix $N\in\N$. We define \begin{align*}R_N(n)&=\sup_{\|f\|_\mu\leq1}\;\sup_{\modu{\varphi(v)}\geq N}\;\mu(v)\modu{(W_{\psi,\varphi}(I-A_n)f)(v)}
		\end{align*}
		and
		\begin{align*}
			S_N(n)&=\sup_{\|f\|_\mu\leq1}\;\sup_{\modu{\varphi(v)}\leq N}\;\mu(v)\modu{(W_{\psi,\varphi}(I-A_n)f)(v)}.
		\end{align*}  Then, from \eqref{Inequality:essentialnormupper} we obtain \[\|W_{\psi,\varphi}\|_e\leq \max\{R_N(n),S_N(n)\}\] for each $n, N\in \N$.  We now consider the case $n >N$. From Lemma \ref{Lemma:compactsequence}, we obtain \begin{align*}
			R_N(n)&=\sup_{\|f\|_\mu\leq1}\;\sup_{\modu{\varphi(v)}\geq N}\;\frac{\mu(v)}{\mu(\varphi(v))}\mu(\varphi(v))\modu{(W_{\psi,\varphi}(I-A_n)f)(v)} \\
			&=\sup_{\|f\|_\mu\leq1}\;\sup_{\modu{\varphi(v)}\geq N}\;\frac{\mu(v)}{\mu(\varphi(v))}\mu(\varphi(v))\modu{\psi(v)}\modu{((I-A_n)f)(\varphi(v))}\\
			&\leq\sup_{\|f\|_\mu\leq1}\;\sup_{\modu{\varphi(v)}\geq N}\;\frac{\mu(v)}{\mu(\varphi(v))}\modu{\psi(v)}\sup_{w\in T}\;\mu(w)\modu{((I-A_n)f)(w)}\\
			&=\sup_{\modu{\varphi(v)}\geq N}\;\frac{\mu(v)}{\mu(\varphi(v))}\modu{\psi(v)}\sup_{\|f\|_\mu\leq1}\;\sup_{w\in T}\;\mu(w)\modu{((I-A_n)f)(w)}\\
			&=\sup_{\modu{\varphi(v)}\geq N}\;\frac{\mu(v)}{\mu(\varphi(v))}\modu{\psi(v)}\|I-A_n\|\\
			&\leq \sup_{\modu{\varphi(v)}\geq N}\;\frac{\mu(v)}{\mu(\varphi(v))}\modu{\psi(v)}.
		\end{align*} Next, observe that
		\begin{align*}
			S_N(n)&=\sup_{\|f\|_\mu\leq1}\;\sup_{\modu{\varphi(v)}\leq N}\;\frac{\mu(v)}{\mu(\varphi(v))}\mu(\varphi(v))\modu{(W_{\psi,\varphi}(I-A_n)f)(v)} \\
			&=\sup_{\|f\|_\mu\leq1}\;\sup_{\modu{\varphi(v)}\leq N}\;\frac{\mu(v)}{\mu(\varphi(v))}\mu(\varphi(v))\modu{\psi(v)}\modu{((I-A_n)f)(\varphi(v))} \\
			&\leq \sup_{\|f\|_\mu\leq1}\;\sup_{\modu{\varphi(v)}\leq N}\;\frac{\mu(v)}{\mu(\varphi(v))}\modu{\psi(v)}\sup_{\modu{w}\leq N}\;\mu(w)\modu{((I-A_n)f)(w)}.
		\end{align*}  If $\modu{w}\leq N$ and $n>N$, then $((I-A_n)f)(w)=0$ and we have $S_N(n)=0$.  Thus, for $n> N$, \[\|W_{\psi,\varphi}\|_e\leq \max\{R_N(n),S_N(n)\}\leq R_N(n)\leq\sup_{\modu{\varphi(v)}\geq N}\;\frac{\mu(v)}{\mu(\varphi(v))}\modu{\psi(v)}.\] This estimate holds for all $N\in \N$, and hence \[\|W_{\psi,\varphi}\|_e\leq \lim_{N\rightarrow\infty} \sup_{\modu{\varphi(v)}\geq N}\;\frac{\mu(v)}{\mu(\varphi(v))}\modu{\psi(v)}.\]
		
		Now assume the essential norm of $W_{\psi,\varphi}$ is strictly less than the limit in \eqref{Equation:essentialnorm}.  Then there is a compact operator $K$ and constant $s>0$ such that \[\|W_{\psi,\varphi}-K\|<s<\lim_{N\rightarrow\infty} \sup_{\modu{\varphi(v)}\geq N}\;\frac{\mu(v)}{\mu(\varphi(v))}\modu{\psi(v)}.\]  Moreover, we can find a sequence of points $(v_n)$ with $\modu{\varphi(v_n)}\rightarrow\infty$ such that \begin{equation}\label{equation:essentialnorm}\limsup_{n\rightarrow\infty}\frac{\mu(v_n)}{\mu(\varphi(v_n))}\modu{\psi(v_n)}>s.\end{equation}  Now, define the sequence of functions $(f_n)$ by \[f_n(v)=\frac{1}{\mu(v)}\bigchi_{\varphi(v_{n})}(v).\]  By \cite[Lemmas 2.4 and 2.5]{AllenPons:2016}, this is a bounded sequence of functions in $\Lmuinf$, with $\|f_n\|_\mu = 1$ for all $n \in \N$, converging to zero pointwise.  We also have the lower estimate, \[s>\|W_{\psi,\varphi}-K\|\geq \|(W_{\psi,\varphi}-K)f_n\|_\mu\geq \|W_{\psi,\varphi} f_n\|_\mu-\|Kf_n\|_\mu.\] By Lemma \ref{Lemma:Compactness-criterion}, $\|Kf_n\|_\mu\rightarrow 0$ as $n \to \infty$, and thus \begin{align*}s &\geq\limsup_{n\rightarrow \infty}\left(\|W_{\psi,\varphi} f_n\|_\mu-\|Kf_n\|_\mu\right)\\
			&=\limsup_{n\rightarrow \infty}\|W_{\psi,\varphi} f_n\|_\mu\\
			&\geq\limsup_{n\rightarrow \infty} \mu(v_n)\modu{\psi(v_n)f_n(\varphi(v_n))}\\
			&=\limsup_{n\rightarrow \infty} \frac{\mu(v_n)}{\mu(\varphi(v_n))}\modu{\psi(v_n)}\\
			&>s,
		\end{align*} which is a contradiction.  Therefore \[\|W_{\psi,\varphi}\|_e = \lim_{N \to \infty} \sup_{\modu{\varphi(v)}\geq N}\;\frac{\mu(v)}{\mu(\varphi(v))}\modu{\psi(v)},\] as desired
	\end{proof}
	
	\begin{corollary}\label{Corollary:Compactness}
		Suppose $\psi$ is a function on $T$ and $\varphi$ is a self-map of $T$ with infinite range for which $W_{\psi,\varphi}$ is bounded on $\Lmuinf$ (respectively $\lilLmuinf$).  Then $W_{\psi,\varphi}$ is compact on $\Lmuinf$ (respectively $\lilLmuinf$) if and only if \[\lim_{N \to \infty} \sup_{\modu{\varphi(v)}\geq N}\; \frac{\mu(v)}{\mu(\varphi(v))}\modu{\psi(v)} = 0.\]
	\end{corollary}
	
	For the spaces constructed with typical weights we can reformulate the essential norm of $W_{\psi,\varphi}$ acting on $\lilLmuinf$ to be a limit superior, and furthermore the characterization of compactness as $\xi_{\psi,\varphi} = 0$.
	
	\begin{theorem}
		Let $\mu$ be a typical weight.  Suppose $\psi$ is a function on $T$ and $\varphi$ is a self-map of $T$ with infinite range for which $W_{\psi,\varphi}$ is bounded on $\lilLmuinf$.  Then \begin{equation}\label{Equation:essential_norm_limsup}\|W_{\psi,\varphi}\|_e = \limsup_{\modu{v} \to \infty} \frac{\mu(v)}{\mu(\varphi(v))}\modu{\psi(v)}.\end{equation}  Moreover, $W_{\psi,\varphi}$ is compact on $\lilLmuinf$ if and only if \[\lim_{\modu{v} \to \infty} \frac{\mu(v)}{\mu(\varphi(v))}\modu{\psi(v)} = 0.\]
	\end{theorem}
	
	\begin{proof}
		To establish the essential norm, by Theorem \ref{Theorem:essential_norm_non-typical} it suffices to show the limit in \eqref{Equation:essentialnorm} is equal to the limit superior in \eqref{Equation:essential_norm_limsup}.  We will first show the limit to be less than or equal to the limit superior.
		
		For $n \in \N$, define \[t_n = \min\{m \in \N : \modu{\varphi(v)} \geq n \text{ for some } v\in T \text{ with } \modu{v} > m\}.\]  We claim $\modu{t_n} \to \infty$ as $n \to \infty$.  If this is not the case, then there exists $N \in \N$ and a sequence of points $(v_n)$ with $\modu{v_n} \leq N$ and $\modu{\varphi(v_n)} \to \infty$.  But this is impossible since $\{v \in T : \modu{v} \leq N\}$ is finite.  Then the set $\{v \in T : \modu{v} > t_n\}$ is precisely \[\{v \in T : \modu{\varphi(v)} \geq n \text{ and } \modu{v} > t_n\} \cup \{v \in T : \modu{\varphi(v)} < n \text{ and } \modu{v} > t_n\}.\]  This implies
		\begin{equation}\label{Inequality:lim+sup<limsup}\begin{aligned}\lim_{N \to \infty} \sup_{\modu{\varphi(v)} \geq N} \frac{\mu(v)}{\mu(\varphi(v))}\modu{\psi(v)} &\leq \lim_{n \to \infty} \sup_{\modu{v} > t_n}\;\frac{\mu(v)}{\mu(\varphi(v))}\modu{\psi(v)}\\ &= \limsup_{\modu{v} \to \infty} \frac{\mu(v)}{\mu(\varphi(v))}\modu{\psi(v)},\end{aligned}\end{equation} where the inequality is due to the fact that $t_n$ is defined as a minimum.
		
		Now we will show equality must hold.  There exists a sequence of vertices $(v_n)$ with $\modu{v_n} \to \infty$ and
		\begin{equation}\label{Inequality:sequencelimsup}\begin{aligned}\lim_{n \to \infty} \frac{\mu(v_n)}{\mu(\varphi(v_n))}\modu{\psi(v_n)} &= \lim_{N \to \infty} \sup_{\modu{v} \geq N} \;\frac{\mu(v)}{\mu(\varphi(v))}\modu{\psi(v)}\\ &= \limsup_{\modu{v} \to \infty} \frac{\mu(v)}{\mu(\varphi(v))}\modu{\psi(v)},\end{aligned}\end{equation} i.e. the limit superior is attained along this sequence.  If the sequence $(\varphi(v_n))$ is bounded, then $\psi \in \lilLmuinf$ (a consequence of the fact that $\mu$ is typical) and \eqref{Inequality:sequencelimsup} imply \[\limsup_{\modu{v} \to \infty} \frac{\mu(v)}{\mu(\varphi(v))}\modu{\psi(v)}=\lim_{n \to \infty} \frac{\mu(v_n)}{\mu(\varphi(v_n))}\modu{\psi(v_n)} = 0.\]  From \eqref{Inequality:lim+sup<limsup} we have \[\lim_{N\to\infty} \sup_{\modu{\varphi(v)}\geq N} \;\frac{\mu(v)}{\mu(\varphi(v))}\modu{\psi(v)} = 0\] as well.
		
		Finally, if $(\varphi(v_n))$ is not bounded, then there exists a subsequence $(v_{n_k})$ with $\modu{v_{n_k}} \to \infty$ and $\modu{\varphi(v_{n_k})} \to \infty$.  Then, by \eqref{Inequality:lim+sup<limsup}, we have \begin{align*}\lim_{k \to \infty} \frac{\mu(v_{n_k})}{\mu(\varphi(v_{n_k}))}\modu{\psi(v_{n_k})} &\leq \lim_{N \to \infty}\sup_{\modu{\varphi(v)} \geq N}\; \frac{\mu(v)}{\mu(\varphi(v))}\modu{\psi(v)}\\&\leq\limsup_{\modu{v} \to \infty} \frac{\mu(v)}{\mu(\varphi(v))}\modu{\psi(v)}\end{align*} since $\modu{\varphi(v_{n_k})} \to \infty.$ From \eqref{Inequality:sequencelimsup}, it follows that equality must hold in this case as well.  The compactness of $W_{\psi,\varphi}$ on $\lilLmuinf$ follows immediately.
	\end{proof}
	
	We complete this section with compactness characteristics for composition operators $C_\varphi$ and multiplication operators $M_\psi$ on $\lilLmuinf$, which were not studied in \cite{AllenPons:2016} or \cite{AllenCraig:2015}.
	
	\begin{corollary} Suppose $\psi$ is a function on $T$ and $\varphi$ a self-map of $T$ for which $C_\varphi$ and $M_\psi$ are bounded on $\lilLmuinf$.
		\begin{enumerate}
			\item[(a)] For the composition operator $C_\varphi$,
			\begin{enumerate}
				\item[i.] if $\varphi$ has finite range, then $C_\varphi$ is compact on $\lilLmuinf$.
				\item[ii.] if $\varphi$ has infinite range, then $C_\varphi$ is compact on $\lilLmuinf$ if and only if \[\lim_{N\to\infty} \sup_{\modu{\varphi(v)}\geq N}\; \frac{\mu(v)}{\mu(\varphi(v))} = 0.\]
				\item[iii.] if $\varphi$ has infinite range and $\mu$ is a typical weight, then $C_\varphi$ is compact on $\lilLmuinf$ if and only if \[\lim_{\modu{v}\to\infty} \frac{\mu(v)}{\mu(\varphi(v))} = 0.\]
			\end{enumerate}
			\item[(b)] The operator $M_\psi$ is compact on $\lilLmuinf$ if and only if \[\lim_{\modu{v}\to\infty} \modu{\psi(v)} = 0.\]
		\end{enumerate}
	\end{corollary}
	
	\section{Boundedness From Below and Closed Range}\label{Section:BoundedBelow}
	Recall a bounded operator $A:X \to Y$ between Banach spaces is bounded below if there exists a positive constant $\delta$ such that $\|Ax\|_Y\geq \delta\|x\|_X$ for all $x \in X$.  As a consequence of the Open Mapping Theorem, a bounded operator $A$ is bounded below if and only if it is injective and has closed range \cite[Proposition VII.6.4]{Conway:85}.  Thus, we first characterize the injective weighted composition operators on $\Lmuinf$ to aid in the characterization of those operators that are bounded below.
	
	To identify the injective weighted composition operators, we define the set $Z = \psi^{-1}(0)$. Recall, for $w \in T$ and $\varphi$ a self-map of $T$, the set $S_w=\varphi^{-1}(\varphi(w)).$
	
	\begin{theorem}\label{lemma:injectiveoperators}
		Let $\psi$ be a function on $T$ and $\varphi$ a self-map of $T$.  Then $W_{\psi,\varphi}$, as an operator on $\Lmuinf$ (respectively $\lilLmuinf$), is injective if and only if $\varphi$ is surjective and for every $w\in T$, $S_w\cap Z^c \neq \emptyset$.
	\end{theorem}
	
	\begin{proof}
		First, suppose $\varphi$ is surjective and for every $w\in T$, $S_w\cap Z^c \neq \emptyset$.  Let $f$ be a function in $\lilLmuinf$ or $\Lmuinf$ that is not the zero function.  Then there is a point $w \in T$ such that $f(w) \neq 0$.  Since $\varphi$ is surjective, there is a $v \in T$ with $\varphi(v) = w$.  From the condition on $S_v$, there is a point $v' \in S_v$ such that $\psi(v') \neq 0$.  Thus
		\[(W_{\psi,\varphi}f)(v') = \psi(v')f(\varphi(v')) = \psi(v')f(\varphi(v)) = \psi(v')f(w) \neq 0.\]  Hence $W_{\psi,\varphi} f$ is not the zero function, and $W_{\psi,\varphi}$ is injective.
		
		For the converse, first suppose $\varphi$ is not surjective.  Then there exists $w \in T$ such that $w \not\in \varphi(T)$.  The function $\bigchi_w(v)$ is a non-zero element of $\lilLmuinf$ and $W_{\psi,\varphi}f = 0$.  Hence, $W_{\psi,\varphi}$ is not injective.
		
		Next, suppose there exists $w \in T$ such that $S_w \subseteq Z$. Then the function $\bigchi_{\varphi(w)}$ is a non-zero element of $\lilLmuinf$, but $W_{\psi,\varphi}\bigchi_{\varphi(w)} = 0$.  Thus, $W_{\psi,\varphi}$ is not injective. In either case, $W_{\psi,\varphi}$ is not injective, completing the proof.
	\end{proof}
	
	To characterize the weighted composition operators that are bounded below, we define the set $U_\varepsilon$ as \[U_\varepsilon=\left\{v\in T : \frac{\mu(v)}{\mu(\varphi(v))}\modu{\psi(v)} \geq \varepsilon\right\}\] for $\psi$ a function on $T$, $\varphi$ a self-map of $T$, and $\varepsilon>0$.
	
	\begin{theorem}\label{theorem:boundedbelowWCO}
		Let $\psi$ be a function on $T$ and $\varphi$ a self-map of $T$ for which $W_{\psi,\varphi}$ is bounded on $\Lmuinf$ (respectively $\lilLmuinf$).  Then $W_{\psi,\varphi}$ is bounded below if and only if $\varphi$ is surjective and there is an $\varepsilon>0$ such that $U_{\varepsilon}\cap S_w \neq \emptyset$ for every $w\in T$.
	\end{theorem}
	
	\begin{proof}
		First, suppose $W_{\psi,\varphi}$ is bounded below.  Then $W_{\psi,\varphi}$ is injective and hence $\varphi$ is surjective by Lemma \ref{lemma:injectiveoperators}.  Also, there is an $\varepsilon_1 >0$ such that $\|W_{\psi,\varphi} f\|_{\mu}\geq \varepsilon_1 \|f\|_{\mu}$ for all $f$ in $\Lmuinf$ or $\lilLmuinf$.  For $w\in T$, take $f(v)=\frac1{\mu(v)}\bigchi_{\varphi(w)}(v)$. Since $\|f\|_{\mu}=1,$ we have $\|W_{\psi,\varphi} f\|_{\mu}\geq \varepsilon_1$ or \[\sup_{v\in T}\;\frac{\mu(v)}{\mu(\varphi(v))}\modu{\psi(v)}\bigchi_{\varphi(w)}(\varphi(v))=\sup_{v\in S_w}\;\frac{\mu(v)}{\mu(\varphi(v))}\modu{\psi(v)}\geq \varepsilon_1.\] Fix $0<\varepsilon<\varepsilon_1$. It follows that for every $w\in T$, there must exist a $v\in S_w$ with \[\frac{\mu(v)}{\mu(\varphi(v))}\modu{\psi(v)}\geq \varepsilon\] and thus $U_\varepsilon\cap S_w\neq \emptyset$.
		
		For the converse, suppose there is an $\varepsilon>0$ such that $U_\varepsilon\cap S_w\neq\emptyset$ for every $w\in T$ and $\varphi$ is surjective. First, let $f \in \Lmuinf$, and observe \[\|W_{\psi,\varphi} f\|_{\mu}=\sup_{v\in T}\;\mu(v)\modu{\psi(v)f(\varphi(v))} =\sup _{v\in T}\;\frac{\mu(v)}{\mu(\varphi(v))}\modu{\psi(v)}\mu(\varphi(v))\modu{f(\varphi(v))}.\]  For $w\in T$, there exists $v\in U_\varepsilon\cap S_w$ and thus \[\varepsilon\mu(\varphi(w))\modu{f(\varphi(w))}= \varepsilon\mu(\varphi(v))\modu{f(\varphi(v))}\leq \frac{\mu(v)}{\mu(\varphi(v))}\modu{\psi(v)}\mu(\varphi(v))\modu{f(\varphi(v))}.\]
		This implies \[\sup_{w\in T}\;\varepsilon\mu(\varphi(w))\modu{f(\varphi(w))}\leq \sup_{v\in T}\;\frac{\mu(v)}{\mu(\varphi(v))}\modu{\psi(v)}\mu(\varphi(v))\modu{f(\varphi(v))}\] or \begin{equation}
			\label{inequality:sneaky}\varepsilon\sup_{w\in T}\;\mu(\varphi(w))\modu{f(\varphi(w))}\leq \|W_{\psi,\varphi} f\|_{\mu}.\end{equation}  Since $\varphi$ is surjective, the supremum on the left is $\|f\|_{\mu}$ and thus, we have $\|W_{\psi,\varphi} f\|_{\mu}\geq \varepsilon\|f\|_{\mu}$ as desired.
	\end{proof}
	
	Considering Theorems \ref{lemma:injectiveoperators} and \ref{theorem:boundedbelowWCO}, it seems natural to expect that $W_{\psi,\varphi}$ has closed range on $\Lmuinf$ or $\lilLmuinf$ if and only if there is an $\varepsilon>0$ such that $U_\varepsilon\cap S_w\neq \emptyset$ for every $w$ for which $\psi(w)\neq 0$.  To verify this claim we will exploit quotient spaces and the fact that an injective operator is bounded below if and only if it has closed range.  The following outlines the necessary details.
	
	Let $X$ be a Banach space and $A:X\to X$ a bounded linear operator.  Then consider the quotient space $X/\ker(A).$  For $x\in X$, \[[x]=x+\ker(A)=\{x+m:m\in\ker(A)\}=\{y\in X: Ax=Ay\}\] and \[\left\|[x]\right\|=\inf\{\|x+m\|:m\in\ker(A)\}.\]  One immediate consequence is that $\left\|[x]\right\|\leq \|x\|.$  Additionally, define an operator $\hat{A}:X/\ker(A)\to X$ by $\hat{A}[x]=Ax.$  This map is well-defined since any $y\in[x]$ satisfies $Ax=Ay$.  It is also easy to see that $\hat{A}$ is linear, injective, and bounded with $\|\hat{A}\|\leq \|A\|.$   Finally, $\textup{range}(A)=\textup{range}(\hat{A}).$  Thus, $A$ has closed range if and only if $\hat{A}$ has closed range. But, since $\hat{A}$ is injective, we know $A$ has closed range if and only if $\hat{A}$ is bounded below.  
	
	\begin{theorem}\label{theorem:WeightedClosedRange}
		Let $\psi$ be a function on $T$ and $\varphi$ a self-map of $T$ for which $W_{\psi,\varphi}$ is bounded on $\Lmuinf$ (respectively $\lilLmuinf$). Then $W_{\psi,\varphi}$ has closed range if and only if there is an $\varepsilon>0$ such that $U_{\varepsilon}\cap S_w\neq \emptyset$ for every $w\in Z^c$.
	\end{theorem}
	
	\begin{proof}
		Suppose $W_{\psi,\varphi}$ has closed range. We will verify the conclusion for $\Lmuinf$, but the same argument suffices for $\lilLmuinf$. Then $\hat{W}_{\psi,\varphi}$ is bounded below  by the discussion above and hence there is an $\varepsilon_1>0$ with $\|W_{\psi,\varphi} f\|_\mu=\|\hat{W}_{\psi,\varphi}[f]\|_\mu\geq \varepsilon_1\|[f]\|_\mu$ for all $f\in\Lmuinf$. For $w\in Z^c$, take $f(v)=\frac1{\mu(v)}\bigchi_{\varphi(w)}(v)$. To estimate $\|[f]\|_\mu$, let $g\in\ker(W_{\psi,\varphi})$.  Then $g(\varphi(v))=0$ for all $v\in Z^c$, which is equivalent to $g(v)=0$ for all $v\in\varphi(Z^c)$.  It follows that \[\begin{aligned}\|f+g\|_\mu=\sup_{v\in T}\mu(v)\modu{f(v)+g(v)}&\geq \sup_{v\in \varphi(Z^c)}\mu(v)\modu{f(v)+g(v)}\\&=\sup_{v\in\varphi(Z^c)}\mu(v)\modu{f(v)}=1\end{aligned}\] and \[\|[f]\|_\mu=\inf\{\|f+g\|_\mu:g\in\ker(C_\varphi)\}\geq 1.\] But $\|[f]\|_\mu\leq\|f\|_\mu=1$ and thus $\|[f]\|_\mu=1.$ From this, for our chosen $f$, we have $\|W_{\psi,\varphi} f\|_\mu=\|\hat{W}_{\psi,\varphi} [f]\|_\mu\geq \varepsilon_1$ or \[\sup_{v\in T}\frac{\mu(v)}{\mu(\varphi(v))}\modu{\psi(v)}\bigchi_{\varphi(w)}(\varphi(v))=\sup_{v\in S_w}\frac{\mu(v)}{\mu(\varphi(v))}\modu{\psi(v)}\geq \varepsilon_1.\] For $0<\varepsilon<\varepsilon_1$, it follows that for every $w\in Z^c$, there must exist a $v\in S_w$ with \[\frac{\mu(v)}{\mu(\varphi(v))}\modu{\psi(v)}\geq \varepsilon\] and thus $U_\varepsilon\cap S_w\neq \emptyset$.
		
		For the converse, suppose there is an $\varepsilon>0$ such that $U_{\varepsilon}\cap S_w\neq \emptyset$ for every $w\in Z^c$. Similarly to the proof of the previous theorem, for $w\in Z^c$, there is a $v\in U_\varepsilon\cap S_w$. From this, for an arbitrary $f\in\Lmuinf$, it follows that \[\varepsilon\mu(\varphi(w))\modu{f(\varphi(w))}= \varepsilon\mu(\varphi(v))\modu{f(\varphi(v))}\leq \frac{\mu(v)}{\mu(\varphi(v))}\modu{\psi(v)}\mu(\varphi(v))\modu{f(\varphi(v))},\] which implies \[\sup_{w\in Z^c}\varepsilon\mu(\varphi(w))\modu{f(\varphi(w))}\leq \sup_{v\in T}\frac{\mu(v)}{\mu(\varphi(v))}\modu{\psi(v)}\mu(\varphi(v))\modu{f(\varphi(v))}\] or \begin{equation}
			\label{inequality:sneaky2}\varepsilon\sup_{w\in Z^c}\mu(\varphi(w))\modu{f(\varphi(w))}\leq \|W_{\psi,\varphi} f\|_{\mu}.\end{equation}  Now, to show the range of $W_{\psi,\varphi}$ is closed, it suffices to show every Cauchy sequence in $\textup{range}(W_{\psi,\varphi})$ has its limit in $\textup{range}(W_{\psi,\varphi})$. Suppose $(W_{\psi,\varphi} f_n)$ is such a Cauchy sequence. First define $g_n=f_n\bigchi_{\varphi(Z^c)}$. Observe that \[(W_{\psi,\varphi} g_n)(v)=\psi(v)f_n(\varphi(v))\bigchi_{\varphi(Z^c)}(\varphi(v))=\psi(v)f_n(\varphi(v))=(W_{\psi,\varphi} f_n)(v)\] for all $v\in T$ and hence $W_{\psi,\varphi} g_n=W_{\psi,\varphi} f_n$  for all $n$.   From (\ref{inequality:sneaky2}), we have \[\begin{aligned}\varepsilon\sup_{w\in Z^c}\mu(\varphi(w))\modu{g_n(\varphi(w))-g_m(\varphi(w))}&\leq \|W_{\psi,\varphi} g_n-W_{\psi,\varphi}g_m\|_{\mu}\\&=\|W_{\psi,\varphi} f_n- W_{\psi,\varphi}f_m\|_{\mu}.\end{aligned}\]  Also, \begin{align*}\sup_{w\in Z^c}\mu(\varphi(w))\modu{g_n(\varphi(w))-g_m(\varphi(w))}&=\sup_{v\in \varphi(Z^c)}\mu(v)\modu{g_n(v)-g_m(v)}\\
			&=\sup_{v\in T}\mu(v)\modu{g_n(v)-g_m(v)}\\
			&=\|g_n-g_m\|_{\mu}\end{align*} where the second equality is due to the definition of the sequence $(g_n)$.  Hence \[\varepsilon\|g_n-g_m\|_{\mu}\leq\|W_{\psi,\varphi} g_n-W_{\psi,\varphi} g_m\|_{\mu}=\|W_{\psi,\varphi} f_n-W_{\psi,\varphi} f_m\|_{\mu},\] which implies $(g_n)$ is a Cauchy sequence in $\Lmuinf$ (resp. $\lilLmuinf$).  Setting $g$ to be the norm limit of $(g_n)$, we have $W_{\psi,\varphi} g=W_{\psi,\varphi}(\lim g_n)=\lim W_{\psi,\varphi} g_n=\lim W_{\psi,\varphi}f_n$ and thus the range of $W_{\psi,\varphi}$ is closed.
	\end{proof}
	
	It is important to note here that many of the results for the operators under investigation have properties similar to those in the continuous setting.  However, in that setting, operators are typically injective by design and hence bounded below if and only if they have closed range.  That is not true in this setting and so the utilization of the technique above is not required in the continuous setting (specifically in the case of analytic function spaces). Therefore this setting highlights the difference between operators that are bounded below and those that have closed range more finely than the continuous setting.
	
	Theorem \ref{lemma:injectiveoperators}, with $\psi \equiv 1$ on $T$, yields a characterization of the injective composition operators on $\Lmuinf$ and $\lilLmuinf$.
	
	\begin{corollary}\label{Corollary:compostiionoperatorinjective}
		Let $\varphi$ be a self-map of $T$.  Then as an operator on $\Lmuinf$ (respectively $\lilLmuinf$), $C_\varphi$ is injective if and only if $\varphi$ is surjective.
	\end{corollary}
	
	Theorems \ref{theorem:boundedbelowWCO}  and  \ref{theorem:WeightedClosedRange} yield characterizations for composition operators that are bounded below or have closed range. In this case, we define the set $V_\varepsilon$ to be
	\[V_\varepsilon=\left\{v\in T : \frac{\mu(v)}{\mu(\varphi(v))} \geq \varepsilon\right\}\] for $\varphi$ a self-map of $T$, and $\varepsilon >0$.
	
	\begin{corollary}\label{Corollary:compositionoperatorboundedbelow}
		Let $\varphi$ be a self-map of $T$, and suppose $C_\varphi$ is bounded on $\Lmuinf$ (respectively $\lilLmuinf$).  Then $C_\varphi$ is bounded below if and only if $\varphi$ is surjective and there is an $\varepsilon>0$ such that $V_{\varepsilon}\cap S_w \neq \emptyset$ for every $w\in T$.
	\end{corollary}
	
	\begin{corollary}\label{Corollary:compositionoperatorclosedrange}
		Let $\varphi$ be a self-map of $T$ and assume $C_\varphi$ is bounded on $\Lmuinf$ (respectively $\lilLmuinf$).  Then $C_\varphi$ has closed range if and only if there is an $\varepsilon>0$ such that $V_{\varepsilon}\cap S_w\neq \emptyset$ for every $w\in T$.  
	\end{corollary}
	
	To characterize the injective multiplication operators on $\Lmuinf$ or $\lilLmuinf$, we can apply Theorem \ref{lemma:injectiveoperators} to the weighted composition operator $W_{\psi,\varphi}$ where $\varphi$ is the identity map on $T$.  In this case, for a point $w \in T$, the set $S_w = \{w\}$.
	
	\begin{corollary}
		Let $\psi$ be a function on $T$.  Then as an operator on $\Lmuinf$ (respectively $\lilLmuinf$), $M_\psi$ is injective if and only if $\psi(v) \neq 0$ for all $v \in T$.
	\end{corollary}
	
	For multiplication operators acting on $\Lmuinf$, a characterization of those that are bounded below was given in \cite[Corollary 3.5]{AllenCraig:2015} using spectral information. Theorem \ref{theorem:boundedbelowWCO} provides a direct proof and extends the result to $\lilLmuinf$.
	
	\begin{corollary}\label{Corollary:multiplicationboundedbelow}
		Let $\psi$ be a function on $T$ and assume $M_\psi$ is bounded on $\Lmuinf$ (respectively $\lilLmuinf$).  Then $M_\psi$ is bounded below if and only if $\inf_{v\in T}\modu{\psi(v)}>0$.
	\end{corollary}
	
	For a multiplication operator to have closed range, 0 can be in the image of $\psi$ but cannot be a limit point; this provides the relevant contrast to Corollary \ref{Corollary:multiplicationboundedbelow}. The result follows immediately from Theorem \ref{theorem:WeightedClosedRange}.
	
	\begin{corollary}\label{Corollary:multiplicationclosedrange}
		Let $\psi$ be a function on $T$ and suppose $M_\psi$ is bounded on $\Lmuinf$ (respectively $\lilLmuinf$).  Then $M_\psi$ has closed range if and only if $\inf_{v\in Z^c}\modu{\psi(v)}>0$.
	\end{corollary}
	
	\section{Invertible and Isometric Weighted Composition Operators}\label{Section:Invertible}
	
	In the next two sections we explore ideas related to those in Section \ref{Section:BoundedBelow} and we restrict our attention to $W_{\psi,\varphi}$ on $\Lmuinf$; some results carry over to $W_{\psi,\varphi}$ on $\lilLmuinf$ with the same proof while other results require more analysis.  We begin with invertibility of weighted composition operators.  Bourdon \cite{Bourdon:2014} noted that when defined, $W_{1/\psi\circ\varphi^{-1},\varphi^{-1}}$ is the inverse of $W_{\psi,\varphi}$.
	
	\begin{theorem}\label{Theorem:InvertibilityTheorem}
		Let $\psi$ be a function on $T$ and $\varphi$ a self-map of $T$ for which $W_{\psi,\varphi}$ is bounded on $\Lmuinf$.  Then $W_{\psi,\varphi}$ has a bounded inverse if and only if $\varphi$ is bijective and \[\inf_{v\in T} \frac{\mu(v)}{\mu(\varphi(v))}\modu{\psi(v)}>0.\] In this case, we have $W_{\psi,\varphi}^{-1} = W_{1/\psi\circ\varphi^{-1},\varphi^{-1}}$ and  \[\|W_{\psi,\varphi}^{-1}\|=\sup_{v\in T}\frac{\mu(v)}{\mu(\varphi^{-1}(v))}\modu{\frac1{\psi(\varphi^{-1}(v))}}=\sup_{v\in T}\frac{\mu(\varphi(v))}{\mu(v)}\modu{\frac1{\psi(v)}}.\]
	\end{theorem}
	
	\begin{proof}
		Suppose $W_{\psi,\varphi}$ has a bounded inverse.  We know $\varphi$ is surjective by Theorem \ref{lemma:injectiveoperators}.  Next assume there is a $v\in T$ with $\psi(v)=0$. Then \[\adWCO K_v=\psi(v)K_{\varphi(v)}=0,\] where $\adWCO$ is the adjoint on the dual space of $\Lmuinf$.  However, this cannot happen since $\adWCO$ is also invertible.  Thus $\psi(v)\neq 0$ for all $v\in T$. To show $\varphi$ is injective, assume $v,w\in T$ with $\varphi(v)=\varphi(w)$.  Then \[\adWCO K_v=\psi(v)K_{\varphi(v)}=\frac{\psi(v)}{\psi(w)}\psi(w)K_{\varphi(w)}=\frac{\psi(v)}{\psi(w)}\adWCO K_w.\] Again using the fact that $\adWCO$ is invertible, we have $K_v=(\psi(v)/\psi(w))K_w$, but this can only happen if $v=w$ by Lemma \ref{lemma:evaluationfunctionalsindep}. We conclude that $\varphi$ is injective and hence bijective.  This conclusion together with the observation that an invertible operator is bounded below and Theorem \ref{theorem:boundedbelowWCO} provides the desired infimum condition.
		
		Conversely, consider the weighted composition operator $W_{1/\psi\circ\varphi^{-1},\varphi^{-1}}$.  The symbols of this operator are defined by our hypotheses on $\psi$ and $\varphi$, and \[\sup_{v\in T}\frac{\mu(v)}{\mu(\varphi^{-1}(v))}\modu{\frac1{\psi(\varphi^{-1}(v))}}=\sup_{v\in T}\frac{\mu(\varphi(v))}{\mu(v)}\modu{\frac1{\psi(v)}}<\infty\] by the infimum condition.  It follows that $W_{1/\psi\circ\varphi^{-1},\varphi^{-1}}$ is bounded by Theorem \ref{Theorem:bounded_criteria_Lmuinfty}, and thus $W_{\psi,\varphi}$ has a bounded inverse.
	\end{proof}
	
	We now focus on characterizing the isometric weighted composition operators acting on $\Lmuinf$.  The characteristic functions give insight into the necessary interplay between $\varphi$ and $\psi$ to induce an isometry.
	
	\begin{theorem}\label{Theorem:WCOIso} Let $\psi$ be a function on $T$ and $\varphi$ a self-map of $T$ for which $W_{\psi,\varphi}$ is bounded.  Then $W_{\psi,\varphi}$ is an isometry on $\Lmuinf$ if and only if $\varphi$ is surjective and $\sup_{v \in S_w}\frac{\mu(v)}{\mu(\varphi(v))}\modu{\psi(v)} = 1$ for all $w \in T$.  Moreover, $W_{\psi,\varphi}$ is a surjective isometry on $\Lmuinf$ if and only if $\varphi$ is a bijection and $\frac{\mu(v)}{\mu(\varphi(v))}\modu{\psi(v)} = 1$ for all $v \in T$.
	\end{theorem}
	
	\begin{proof}
		We first prove the characterization of the isometric weighted composition operators on $\Lmuinf$. Suppose $W_{\psi,\varphi}$ is an isometry on $\Lmuinf$.  Then $W_{\psi,\varphi}$ is injective, and thus $\varphi$ is surjective by Theorem \ref{lemma:injectiveoperators}.  Fix $w \in T$ and consider the function $f(v) = \frac{1}{\mu(v)}\bigchi_{\varphi(w)}(v)$.  Since $W_{\psi,\varphi}$ is an isometry on $\Lmuinf$, it follows that
		\[1 = \|f\|_\mu = \|W_{\psi,\varphi}f\|_\mu = \sup_{v \in T} \frac{\mu(v)}{\mu(\varphi(v))}\modu{\psi(v)}\bigchi_{\varphi(w)}(\varphi(v)) = \sup_{v \in S_w} \frac{\mu(v)}{\mu(\varphi(v))}\modu{\psi(v)}.\]
		
		Conversely, suppose $\varphi$ is surjective and $\sup_{v \in S_w}\frac{\mu(v)}{\mu(\varphi(v))}\modu{\psi(v)} = 1$ for all $w \in T$.  Observe that for each $v \in T$, $\frac{\mu(v)}{\mu(\varphi(v))}\modu{\psi(v)} \leq 1$ since $v \in S_v$.  Let $f \in \Lmuinf$.  It follows from Lemma \ref{Lemma:point_evaluation_bound} that
		\[\|W_{\psi,\varphi}f\|_\mu = \sup_{v \in T} \mu(v)\modu{\psi(v)}\modu{f(\varphi(v))} \leq \sup_{v \in T} \frac{\mu(v)}{\mu(\varphi(v))}\modu{\psi(v)}\|f\|_\mu \leq \|f\|_\mu.\]
		To verify the reverse inequality, fix $w \in T$ and choose $0 < \varepsilon < 1$.  Then there exists $v \in U_\varepsilon\cap S_w$.
		Observe 
		\begin{align*}
			\varepsilon\mu(\varphi(w))\modu{f(\varphi(w))} &= \varepsilon\mu(\varphi(v))\modu{f(\varphi(v))}\\
			&\leq \frac{\mu(v)}{\mu(\varphi(v))}\modu{\psi(v)}\mu(\varphi(v))\modu{f(\varphi(v))}\\
			&= \mu(v)\modu{\psi(v)}\modu{f(\varphi(v))}\\
			&\leq \|W_{\psi,\varphi} f\|_\mu.
		\end{align*} Taking the supremum over all $w \in T$, and letting $\varepsilon$ go to 1, we have $\|f\|_\mu \leq \|W_{\psi,\varphi} f\|_\mu$.  Thus, $W_{\psi,\varphi}$ is an isometry on $\Lmuinf$.
		
		We complete the proof by characterizing the surjective isometric weighted composition operators on $\Lmuinf$.  Suppose $W_{\psi,\varphi}$ is a surjective isometry.  Then $W_{\psi,\varphi}$ has bounded inverse and it follows from Theorem \ref{Theorem:InvertibilityTheorem} that $\varphi$ is bijective.  Thus $S_w = \{w\}$ for every $w \in T$ and \[\frac{\mu(w)}{\mu(\varphi(w))}\modu{\psi(w)} = \sup_{v \in S_w} \frac{\mu(v)}{\mu(\varphi(v))}\modu{\psi(v)} = 1\] for every $w \in T$.
		
		Finally, suppose $\varphi$ is a bijection and $\frac{\mu(v)}{\mu(\varphi(v))}\modu{\psi(v)} = 1$ for all $v \in T$.  Then $W_{\psi,\varphi}$ is an isometry.  It follows from Theorem \ref{Theorem:InvertibilityTheorem} that $W_{\psi,\varphi}$ is invertible, and thus surjective.
	\end{proof}
	
	We complete this section by considering the results applied to the composition and multiplication operators acting on $\Lmuinf$.
	
	\begin{corollary}\label{Corollary:compositioninvertible}
		Suppose $\psi$ is a function on $T$ and $\varphi$ a self-map of $T$ for which both $C_\varphi$ and $M_\psi$ are bounded on $\Lmuinf$.  
		\begin{enumerate}
			\item[(a)] Then $C_\varphi$ has a bounded inverse if and only if $\varphi$ is bijective and \[\inf_{v\in T} \frac{\mu(v)}{\mu(\varphi(v))}>0.\]  In this case, we have \[\|C_\varphi^{-1}\|=\|C_{\varphi^{-1}}\|=\sup_{v\in T}\frac{\mu(v)}{\mu(\varphi^{-1}(v))}=\sup_{v\in T}\frac{\mu(\varphi(v))}{\mu(v)}.\]
			\item[(b)] Then $M_\psi$ has a bounded inverse if and only if $\inf_{v\in T} \modu{\psi(v)}>0$.  In this case, we have \[\|M_\psi^{-1}\| = \|M_{1/\psi}\| = \|1/\psi\|_\infty.\]
		\end{enumerate}
	\end{corollary}
	
	The following characterization of the isometric composition operators on $\Lmuinf$ completes the work of the authors in \cite{AllenPons:2016}, where only partial results were obtained for $\Lmuinf$.  The characterization of the isometric multiplication operators on $\Lmuinf$ agrees with \cite[Theorem 3.6]{AllenCraig:2015}, while providing the additional conclusion that all such isometries are in fact surjective.
	
	\begin{corollary}
		Suppose $\psi$ is a function on $T$ and $\varphi$ a self-map of $T$ for which both $C_\varphi$ and $M_\psi$ are bounded on $\Lmuinf$.
		\begin{enumerate}
			\item[(a)] The operator $C_\varphi$ is an isometry on $\Lmuinf$ if and only if $\varphi$ is a surjective and $\sup_{v \in S_w}\frac{\mu(v)}{\mu(\varphi(v))} = 1$ for all $w \in T$.  
			\item[(b)] The operator $C_\varphi$ is a surjective isometry on $\Lmuinf$ if and only if $\varphi$ is a bijection and $\frac{\mu(v)}{\mu(\varphi(v))} = 1$ for all $v \in T$.
			\item[(c)] For the multiplication operator $M_\psi$, the following are equivalent:
			\begin{enumerate}
				\item[i.] $M_\psi$ is an isometry on $\Lmuinf$,
				\item[ii.] $M_\psi$ is a surjective isometry on $\Lmuinf$,
				\item[iii.] $\modu{\psi(v)} = 1$ for all $v \in T$.
			\end{enumerate}
		\end{enumerate}
	\end{corollary}
	
	\section{Fredholm Weighted Composition Operators}\label{Section:Fredholm}
	Recall a linear operator $A$ between Banach spaces is Fredholm if $A$ has closed range and both $\ker(A)$ and $\ker(A^*)$ are finite dimensional.  In fact, the condition of closed range is redundant, since this follows from the dimension of the cokernel being finite.  However, this condition typically remains to mirror the definition of Fredholm operators on a Hilbert space.  Alternatively, $A$ is Fredholm if there exists a bounded operator $S$ such that $SA-I$ and $AS-I$ are both compact.  This is sometimes referred to as Atkinson's Theorem.  Every invertible operator is Fredholm.  The converse is not true, but an operator that is Fredholm must be invertible ``modulo the compacts".  Thus a compact operator can not be Fredholm.  For a general reference on Fredholm operators see \cite[Section XI.2]{Conway:85} or \cite[Section 5.8]{MacCluer:09}.
	
	To classify the Fredholm weighted composition operators, we begin with a sequence of useful lemmas.
	
	\begin{lemma}\label{lemma:psifinitezeros}
		Let $\psi$ be a function on $T$ and $\varphi$ a self-map of $T$ for which $W_{\psi,\varphi}$ is bounded on $\Lmuinf$.  If $W_{\psi,\varphi}$ is Fredholm on $\Lmuinf$, then $\psi$ can have at most finitely many zeros. 
	\end{lemma}
	
	\begin{proof}
		First, we show $K_w\in\ker(\adWCO)$ whenever $\psi(w)=0$.  Suppose $w\in T$ with $\psi(w)=0$ and let $f\in\Lmuinf$. Then \[(\adWCO K_wf)(v)=K_w(\psi(v)f(\varphi(v)))=\psi(w)f(\varphi(w))=0.\] Since $f$ was arbitrary, this implies $\adWCO K_w$ is the zero functional and thus $K_w\in\ker(\adWCO)$.
		
		If $\psi$ has infinitely many zeros $\{v_i\}_{i=1}^{\infty}$, then $\{K_{v_i}\}_{i=1}^{\infty}\subseteq \textup{ker}(\adWCO).$ From Lemma \ref{lemma:evaluationfunctionalsindep}, the set $\{K_{v_i}\}$ is linearly independent and thus $\dim(\ker(\adWCO))=\infty.$  However, since $W_{\psi,\varphi}$ is Fredholm, $\dim(\ker(\adWCO))<\infty$.  This is a contradiction and the result follows.
	\end{proof}
	
	\begin{lemma}
		Let $\psi$ be a function on $T$ and $\varphi$ a self-map of $T$. If $\varphi$ has finite range and $W_{\psi,\varphi}$ is bounded on $\Lmuinf$, then $W_{\psi,\varphi}$ cannot be Fredholm.
	\end{lemma}
	
	\begin{proof}
		This follows from the fact that such a weighted composition operator is compact, and hence cannot be Fredholm.
	\end{proof}
	
	\begin{lemma}\label{lemma:pullbackbound}
		Let $\psi$ be a function on $T$ and $\varphi$ a self-map of $T$ for which $W_{\psi,\varphi}$ is bounded on $\Lmuinf$. If $W_{\psi,\varphi}$ is Fredholm on $\Lmuinf$, then there is an $N\in\N$ such that $\varphi^{-1}(w)$ contains at most $N$ points for every $w\in \varphi(T)$.
	\end{lemma}
	
	\begin{proof}
		First, suppose there exist points $\{v_1, \dots, v_{2m}\}$ in $T$ such that $\varphi(v_1) = \cdots = \varphi(v_{2m})$ and $\psi(v_i) \neq 0$ for $1 \leq i \leq 2m$.  For $1 \leq j \leq m$, define $k_j = \frac{1}{\psi(v_{2j})}K_{v_{2j}} - \frac{1}{\psi(v_{2j-1})}K_{v_{2j-1}}.$ It follows that $\{k_1,\dots,k_m\} \subseteq \ker(\adWCO)$.
		
		Now, if  $W_{\psi,\varphi}$ is Fredholm and the condition of the lemma does not hold, then for every $n$ there is a $w_n\in \varphi(T)$ such that $\varphi^{-1}(w_n)$ contains at least $n$ elements. Furthermore, $\psi$ has at most finitely many zeros by Lemma \ref{lemma:psifinitezeros}, say $M$, and thus for $n$ sufficiently large $\varphi^{-1}(w_n)$ contains at least $n-M>0$ points where $\psi$ does not vanish. Letting $n$ tend to infinity, the first part of the proof implies $\dim(\ker(\adWCO))=\infty$. However this contradicts the fact that $W_{\psi,\varphi}$ is Fredholm.
	\end{proof}
	
	\begin{lemma} \label{lemma:singleton}
		Let $\psi$ be a function on $T$ and $\varphi$ a self-map of $T$ for which $W_{\psi,\varphi}$ is bounded on $\Lmuinf$. If $W_{\psi,\varphi}$ is Fredholm on $\Lmuinf$, then $\varphi^{-1}(w)$ contains exactly one element for all but a finite number of points $w\in\varphi(T)$.
	\end{lemma}
	
	\begin{proof}
		Suppose $W_{\psi,\varphi}$ is Fredholm and define \[K=\{w\in\varphi(T): \varphi^{-1}(w) \textup{ contains more than one element}\}.\]  Assume to the contrary that $K$ is an infinite set.  From Lemma 7.1 it follows that $Z = \psi^{-1}(0)$ is finite.  Thus, there is an infinite subset $K_1\subseteq K$ such that for every $w\in K_1$, there exists two distinct points $v_1,v_2\in\varphi^{-1}(w)$ for which $\psi(v_1) \neq 0$ and $\psi(v_2) \neq 0$.  Define $k_w=
		\frac{1}{\psi(v_1)}K_{v_1} - \frac{1}{\psi(v_2)}K_{v_2}$ and let $S = \{k_w : w \in K_1\}$. Note $S$ is infinite and linearly independent. A computation similar to the one above shows $S\subseteq \ker(\adWCO).$  Again, this contradicts the fact that $W_{\psi,\varphi}$ is Fredholm and the conclusion follows.
	\end{proof}
	
	\begin{lemma} \label{lemma:imagefinite}
		Let $\psi$ be a function on $T$ and $\varphi$ a self-map of $T$ for which $W_{\psi,\varphi}$ is bounded on $\Lmuinf$. If $W_{\psi,\varphi}$ is Fredholm on $\Lmuinf$, then $T\setminus\varphi(T)$ must be finite.
	\end{lemma}
	
	\begin{proof}
		Assume $T\setminus\varphi(T)$ is infinite.  First notice $\{\bigchi_{w}:w\in T\setminus\varphi(T)\}$ is linearly independent in $\Lmuinf$.  Also, for $w\in T\setminus\varphi(T)$, we have \[(W_{\psi,\varphi} \bigchi_w)(v)=\psi(v)\bigchi_w(\varphi(v))=0.\] Thus $\{\bigchi_{w}:w\in T\setminus\varphi(T)\}\subseteq \ker(W_{\psi,\varphi})$. However, this contradicts the fact that $\dim(\ker(W_{\psi,\varphi}))<\infty$.
	\end{proof}
	
	The last lemma we need is derived from Theorems \ref{Theorem:bounded_criteria_Lmuinfty} and \ref{Theorem:InvertibilityTheorem}.
	
	\begin{lemma}\label{lemma:restrictionnorm} Let $X$ and $Y$ be unbounded subsets of $T$.  Suppose $\psi$ is a function on $T$ and $\varphi$ a self-map of $T$ for which $W_{\psi,\varphi}$ is bounded on $\Lmuinf(T)$.  If $\varphi:X \to Y$, then $W_{\psi,\varphi}:\Lmuinf(Y) \to \Lmuinf(X)$ is bounded and satisfies \[\|W_{\psi,\varphi}\| = \sup_{v \in X} \frac{\mu(v)}{\mu(\varphi(v))}\modu{\psi(v)}.\]
		Moreover, if $\varphi$ is bijective and $\inf_{v\in X} \frac{\mu(v)}{\mu(\varphi(v))}\modu{\psi(v)}>0,$  then $W_{\psi,\varphi}$ has bounded inverse with \[\|W_{\psi,\varphi}^{-1}\|=\sup_{v\in Y}\;\frac{\mu(v)}{\mu(\varphi^{-1}(v))}\modu{\frac{1}{\psi(\varphi^{-1}(v))}}=\sup_{v\in X}\frac{\mu(\varphi(v))}{\mu(v)}\modu{\frac{1}{\psi(v)}}.\]
	\end{lemma}
	
	Combining the previous lemmas leads to our Fredholm characterization. 
	For the proof here, recall the notation 
	\begin{align*}
		S_w &= \{v \in T: \varphi(v) = \varphi(w)\}\\
		U_\varepsilon &= \left\{v \in T: \frac{\mu(v)}{\mu(\varphi(v))}\modu{\psi(v)} \geq \varepsilon\right\}.
	\end{align*}
	
	\begin{theorem}\label{theorem:WCOFredholm} Let $\psi$ be a function on $T$ and $\varphi$ be a self-map of $T$ for which $W_{\psi,\varphi}$ is bounded on $\Lmuinf$.  Then $W_{\psi,\varphi}$ is Fredholm if and only if 
		\begin{enumerate}
			\item[(a)] $T\setminus\varphi(T)$ is finite,
			\item[(b)] there exists $M \in \N$ such that $\varphi^{-1}(w)$ contains at most $M$ points for every $w \in \varphi(T)$,
			\item[(c)] $\varphi^{-1}(w)$ contains exactly one element for all but a finite number of points $w \in \varphi(T)$,
			\item[(d)] $Z = \{v \in T : \psi(v) = 0\}$ is finite,
			\item[(e)] there is an $\varepsilon>0$ such that $U_{\varepsilon}\cap S_w \neq \emptyset$ for every $w\in Z^c$.
		\end{enumerate}
	\end{theorem}
	
	\begin{proof}
		First, suppose $W_{\psi,\varphi}$ is Fredholm.  Then properties (a) through (e) follow directly from Lemmas \ref{lemma:imagefinite}, \ref{lemma:pullbackbound}, \ref{lemma:singleton}, \ref{lemma:psifinitezeros}, and Theorem \ref{theorem:WeightedClosedRange} (since Fredholm implies closed range).
		
		Now, suppose conditions (a) through (e) hold.  To prove $W_{\psi,\varphi}$ is Fredholm, we define the following sets 
		\begin{align*}
			\mathcal{S} &= \{w \in \varphi(T) : \modu{\varphi^{-1}(w)} = 1\}\\
			\mathcal{F} &= \{w \in \varphi(T) : 1 < \modu{\varphi^{-1}(w)} \leq M\}.
		\end{align*} Note $\varphi(T) = \mathcal{S} \cup \mathcal{F}$ from condition (b).  Also, by condition (c), $\mathcal{F}$ is a finite set, which we enumerate as $\mathcal{F} = \{w_1, \dots, w_m\}$.  For each $1 \leq i \leq m$, we denote $v_i$ to be a fixed element in the set $\varphi^{-1}(w_i)$.
		
		Next, define the function $\eta:T \to T$ by 
		\[\eta(w) = \begin{cases} \varphi^{-1}(w) & \text{if $w \in S$}\\v_i & \text{if $w \in \mathcal{F}$}\\ w & \text{if $w \not\in \varphi(T)$}\end{cases}\] and the function $\tau:T \to \Co$ by \[\tau(w) = \begin{cases}
			1/\psi(v) & \text{if $v \in Z^c$}\\
			1 & \text{if $v \in Z$}.\\
		\end{cases}\] Note that by properties (a) through (c), the set $\{v \in T : \eta(v) \in Z\}$ is finite by construction since $Z$ is finite.  
		
		We will now show $W_{\tau\circ\eta,\eta}$ to be bounded as an operator on $\Lmuinf$.  First, define $S_1 = \mathcal{S} \cap \varphi(Z^c)$ and $T_1 = \varphi^{-1}(S_1)$.  Then, as a mapping, $\varphi:T_1\to S_1$ is bijective.  Also, by condition (c), $S_1$ is infinite and hence unbounded; it follows then that $T_1$ is also unbounded.  Restricting $\eta$ to $S_1$, we have $\eta =\varphi^{-1}$ and $\tau\circ\eta = 1/\psi\circ\varphi^{-1}$.  By Lemma \ref{lemma:restrictionnorm} and condition (e), $W_{\psi,\varphi}$ has a bounded inverse $W_{1/\psi\circ\varphi^{-1},\varphi^{-1}}$ as an operator from $\Lmuinf(S_1)$ to $\Lmuinf(T_1)$.  But this inverse is $W_{\tau\circ\eta, \eta}$.  As $T\setminus S_1 = \mathcal{S}^c \cup (T\setminus\varphi(T)) \cup \varphi(Z)$ is the union of finite sets by conditions (a), (c), and (d), and thus is a finite set, $W_{\tau\circ\eta,\eta}$ extends to a bounded operator on $\Lmuinf(T)$.
		
		To show $W_{\psi,\varphi}$ is Fredholm, we will show $W_{\psi,\varphi} W_{\tau\circ\eta,\eta} - I$ and $W_{\tau\circ\eta,\eta} W_{\psi,\varphi} - I$ are both compact.  Observe that \[\varphi(\eta(v)) = \begin{cases} v & \text{if $v \in \varphi(T)$}\\\varphi(v) & \text{if $v \not\in \varphi(T)$}.\end{cases}\] and for $f \in \Lmuinf$ and $v \in T$,
		\[(W_{\tau\circ\eta,\eta} W_{\psi,\varphi} f)(v) = \begin{cases}
			f(v) & \text{if $v \in \varphi(T)$ and $\eta(v) \in Z^c$}\\
			f(\varphi(v)) & \text{if $v \notin \varphi(T)$ and $\eta(v) \in Z^c$}\\
			0 & \text{if $\eta(v) \in Z$}.
		\end{cases}\]
		So \[((W_{\tau\circ\eta,\eta} W_{\psi,\varphi} -I)f)(v) = \sum_{\substack{w\notin \varphi(T) \\ \eta(w)\in Z^c}} (f(\varphi(w))-f(w))\bigchi_w(v) - \sum_{\substack{w\in T \\ \eta(w)\in Z}} f(w)\bigchi_w(v).\]  As these are finite sums, the operator $W_{\tau\circ\eta,\eta}W_{\psi,\varphi}  - I$ is finite-rank, and thus compact.  Likewise, \[\eta(\varphi(v)) = \begin{cases} v & \text{if $\varphi(v) \in \mathcal{S}$}\\v_i & \text{if $\varphi(v) \in \mathcal{F}$}\end{cases}\] and for $f \in \Lmuinf$ and $v \in T$
		\[(W_{\psi,\varphi} W_{\tau\circ\eta,\eta} f)(v) = \begin{cases}
			f(v) & \text{if $\varphi(v) \in S \cap \varphi(Z^c)$}\\
			\psi(v)\tau(v_i)f(v_i) & \text{if $\varphi(v) \in F \cap \varphi(Z^c)$}\\
			0 & \text{if $\varphi(v) \in \varphi(Z)$}.
		\end{cases}\]
		So $W_{\psi,\varphi}W_{\tau\circ\eta,\eta} - I$ is compact since
		\[\begin{aligned}&((W_{\psi,\varphi} W_{\tau\circ\eta,\eta} -I)f)(v)=\\ 
			&\qquad\sum_{\substack{w\in T \\ \varphi(w)\in \mathcal{F}\cap\varphi(Z^c)}} (\psi(w)\tau(v_i)f(v_i)-f(w))\bigchi_{w}(v) - \sum_{\substack{w\in T \\ \varphi(w)\in\varphi(Z)}} f(w)\bigchi_w(v).\end{aligned}\]  Therefore $W_{\psi,\varphi}$ is Fredholm.
	\end{proof}
	
	In the case when $\varphi$ is a bijection on $T$, we see that the Fredholm condition is almost the same as that for bounded below.
	
	\begin{corollary} Let $\psi$ be a function on $T$ and $\varphi$ be a bijective self-map of $T$ and assume $W_{\psi,\varphi}$ is bounded on $\Lmuinf$.  Then $W_{\psi,\varphi}$ is Fredholm if and only if $W_{\psi,\varphi}$ is bounded below and $\psi$ has finitely many zeros.
	\end{corollary}
	
	For composition operators we have the following.  Recall the notation \[V_\varepsilon = \left\{v \in T : \frac{\mu(v)}{\mu(\varphi(v))} \geq \varepsilon\right\}.\]
	
	\begin{corollary}\label{theorem:compositionFredholm} Let $\varphi$ be a self-map of $T$ for which $C_\varphi$ is bounded on $\Lmuinf$.  Then $C_\varphi$ is Fredholm if and only if 
		\begin{enumerate}
			\item[(a)] $T\setminus\varphi(T)$ is finite,
			\item[(b)] there exists $M \in \N$ such that $\varphi^{-1}(w)$ contains at most $N$ points for every $w \in \varphi(T)$,
			\item[(c)] $\varphi^{-1}(w)$ contains exactly one element for all but a finite number of points $w \in \varphi(T)$,
			\item[(d)] there is an $\varepsilon>0$ such that $V_{\varepsilon}\cap S_w \neq \emptyset$ for every $w\in T$.
		\end{enumerate}
	\end{corollary}
	
	In addition to the characterization above, for composition operators  we have an interesting sufficient condition for an operator to be Fredholm.
	
	\begin{proposition}\label{proposition:SufficientforFredholmComposition}
		Let $\varphi$ be self-map of $T$ for which $C_\varphi$ is bounded on $\Lmuinf$.  If there is a self-map $\eta$ of $T$ such that $\varphi(v)=\eta(v)$ except for a finite number of points in $T$ and $C_{\eta}$ is invertible (with bounded inverse), then $C_\varphi$ is Fredholm.
	\end{proposition}
	
	\begin{proof}
		First observe that $C_\eta$ is bounded on $\Lmuinf$ since $C_\varphi$ is bounded. The fact that $C_\eta$ is invertible implies $\eta$ is surjective by Theorem \ref{lemma:injectiveoperators}.  Furthermore, $C_\eta^*$ is also invertible.  If $v,w\in T$ with $\eta(v)=\eta(w)$, it must be the case that $K_{\eta(v)}=K_{\eta(w)}$ or $C_\eta^*K_v=C_\eta^*K_w$ and thus $K_v=K_w$ or $v=w$.  Thus $\eta$ is injective and hence invertible.  It is easy to check that $C_\eta^{-1}=C_{\eta^{-1}}$.
		
		Let $R=\{v_1,v_2,\ldots,v_n\}$ be the set of points where $\varphi$ and $\eta$ do not agree. We now claim that $\textup{Im}(C_\varphi C_\eta^{-1}-I)$ and $\textup{Im}(C_\eta^{-1}C_\varphi-I)$ are both finite dimensional. This in turn implies that both operators are finite rank and hence compact. We verify these claims from which the conclusion is apparent.  
		
		For $C_\varphi C_\eta^{-1}-I=C_\varphi C_{\eta^{-1}}-I$ and $v\in T$, we have $((C_\varphi C_{\eta^{-1}}-I)f)(v)=f(\eta^{-1}(\varphi(v)))-f(v)$, which will be zero for $v\not\in R$.  Thus we can write \[((C_\varphi C_{\eta^{-1}}-I)f)(v)=\sum_{v_i\in R} (f(\eta^{-1}(\varphi(v_i)))-f(v_i))\bigchi_{v_i}(v).\]  Thus $\textup{Im}(C_\varphi C_\eta^{-1}-I)\subseteq \left\{\sum c_i\bigchi_{v_i}: c_i\in \Co, v_i\in R\right\}$, which is finite dimensional.
		
		Similarly, $((C_{\eta^{-1}}C_\varphi-I)f)(v)=f(\varphi(\eta^{-1}(v)))-f(v)$, which will be zero if $\eta^{-1}(v)\not\in R$. If $\eta^{-1}(v)\in R$, then $\eta^{-1}(v)=v_i$ for some $i$ and thus, in this case, we have \[((C_{\eta^{-1}}C_\varphi-I)f)(v)=\sum_{v_i\in R} (f(\varphi(v_i))-f(\eta(v_i)))\bigchi_{\eta(v_i)}(v).\]  We conclude $\textup{Im}(C_\eta^{-1}C_\varphi-I)\subseteq\left\{\sum c_i\bigchi_{\eta(v_i)}: c_i\in \Co, v_i\in R\right\}$.
	\end{proof}
	
	For multiplication operators, we immediately see the following.
	
	\begin{corollary}\label{theorem:multiplicationFredholm}
		Let $\psi$ be a function on $T$ for which $M_\psi$ is bounded on $\Lmuinf$.  The following are equivalent:
		
		\begin{enumerate}
			\item $M_\psi$ is Fredholm,
			\item $Z=\{v\in T:\psi(v)=0\}$ is finite and $\inf_{v\in Z^c}\modu{\psi(v)}>0$,
			\item there exists $\varepsilon>0$ and $N\in \N$ such that $\modu{\psi(v)}\geq\varepsilon$ for all $v\in T$ with $\modu{v}>N$.
		\end{enumerate} 
	\end{corollary} 
	\noindent Note condition (ii) is derived from Theorem \ref{theorem:WCOFredholm}.  In addition, condition (iii), while easily shown to be equivalent to (ii), is a simpler condition to use for actually determining if $M_\psi$ is Fredholm on $\Lmuinf$.
	
	\section{Examples}\label{Section:Examples}
	In this section, we construct several examples that illustrate the richness of the operator theory as well as several key features of the results throughout this paper.  Specifically, when not indicated, $T$ will be an infinite tree with root $o$, where $\mathrm{d}$ is the edge-counting metric on $T$, as defined in any of \cite{AllenColonnaEasley:14,AllenCraig:2015,AllenPons:2016,ColonnaEasley:10}. 
	For the first example, we construct a weighted composition operator that is bounded on $\Lmuinf$ but not bounded on $\lilLmuinf$, thus showing the converse of Lemma \ref{Lemma:Bounded_inf_0} does not hold.  
	
	\begin{example} Let $\mu$ be a weight function and $\varphi$ a self-map of $T$ with finite range.  Define $\psi(v) = 1/\mu(v)$ for all $v \in T$, $m = \min_{w \in \varphi(T)} \psi(w)$, and $M = \max_{w \in \varphi(T)} \psi(w)$.  Note $0 < m \leq M < \infty$ and  \[\frac{\mu(v)}{\mu(\varphi(v))}\modu{\psi(v)} = \psi(\varphi(v))\] for all $v \in T$.  Thus $\sigma_{\psi,\varphi} = M$. So $W_{\psi,\varphi}$ is bounded on $\Lmuinf$ by Theorem \ref{Theorem:bounded_criteria_Lmuinfty}.  However, since $\psi(\varphi(v))$ is bounded away from zero on $T$, it follows that $\xi_{\psi,\varphi}$ can not equal 0.  Thus $W_{\psi,\varphi}$ is not bounded on $\lilLmuinf$ by Theorem \ref{Theorem:bounded_criteria_Lmu0_elliptic}.
	\end{example}
	
	In the next three examples, we construct bounded and compact weighted composition operators on $\Lmuinf$ for which the induced multiplication or composition operators are not bounded or compact.
	
	\begin{example} In this example, we construct a bounded weighted composition operators on $\Lmuinf$ for which the induced composition operator is bounded but the induced multiplication operator is unbounded.  We specifically provide separate examples for which $\mu$ is typical and atypical.
		\begin{enumerate}
			\item Define \[\mu(v) = \begin{cases}1/\modu{v} & \text{if $v \neq o$}\\1 & \text{if $v = o$},\end{cases}\] $\psi(v) = 1/\mu(v)$ and $\varphi(v) = o$ for all $v \in T$.  Since $\psi$ is not a bounded function on $T$, $M_\psi$ is not bounded on $\Lmuinf$ \cite[Theorem 3.1]{AllenCraig:2015}. Additionally, $C_\varphi$ is bounded (it is in fact compact) \cite[Theorem 3.1]{AllenPons:2016} since \[\frac{\mu(v)}{\mu(\varphi(v))} = \frac{1}{\modu{v}}\] for all $v \in T^*$.  However, $W_{\psi,\varphi}$ is bounded on $\Lmuinf$ since \[\frac{\mu(v)}{\mu(\varphi(v))}\modu{\psi(v)} = 1\]  for all $v \in T$.
			\item Define \[\mu(v) = \begin{cases}\modu{v} & \text{if $v \neq o$}\\1 & \text{if $v = o$},\end{cases}\] and $\psi(v) = \mu(v)$ for all $v \in T$.  Let $(w_n)$ be a sequence in $T$ for which $\modu{w_n} = n^2$ for all $n \in \N\cup\{0\}$ and define $\varphi(v) = w_{\modu{v}}$ for all $v \in T$.  For the same reasoning as in (i), $M_\psi$ is not bounded and $C_\varphi$ is bounded on $\Lmuinf$, and $W_{\psi,\varphi}$ is bounded on $\Lmuinf$.
		\end{enumerate}
	\end{example}
	
	\begin{example} In this example, we construct a compact weighted composition operator on $\Lmuinf$ for which the induced multiplication operator is compact but the induced composition operator is not bounded.  We provide specific examples for which $\mu$ is typical and atypical.
		
		\begin{enumerate}
			\item Define \[\mu(v) = \begin{cases}1/\modu{v} & \text{if $v \neq o$}\\1 & \text{if $v = o$},\end{cases}\] and $\psi(v) = \mu^3(v)$ for all $v \in T$.  Let $(w_n)$ be a sequence in $T$ for which $\modu{w_n} = n^2$ for all $n \in \N\cup\{0\}$ and define $\varphi(v) = w_{\modu{v}}$ for all $v \in T$.  Observe $M_\psi$ is compact on $\Lmuinf$ \cite[Theorem 3.2]{AllenCraig:2015} since $\mu^3(v) \to 0$ as $\modu{v} \to \infty$.  Also note $C_\varphi$ is not bounded on $\Lmuinf$ since \[\frac{\mu(v)}{\mu(\varphi(v))} = \modu{v} \to \infty\] as $\modu{v} \to \infty$. However, $W_{\psi,\varphi}$ is compact on $\Lmuinf$ by Corollary \ref{Corollary:Compactness} since \[\sup_{\modu{\varphi(v)}\geq N}\;\frac{\mu(v)}{\mu(\varphi(v))}\modu{\psi(v)} = \frac{1}{N}.\]
			
			\item Define \[\mu(v) = \begin{cases}\modu{v}^2 & \text{if $v \neq o$}\\1 & \text{if $v = o$},\end{cases}\] and $\psi(v) = 1/\mu(v)$ for all $v \in T$.  Let $(w_n)$ be a sequence in $T$ for which $\modu{w_n} = n$ for all $n \in \N\cup\{0\}$ and define $\varphi(v) = w_{\modu{v}}$ for all $v \in T$.  For the same reasoning as in (i), $M_\psi$ is compact and $C_\varphi$ is not bounded on $\Lmuinf$, and $W_{\psi,\varphi}$ is compact on $\Lmuinf$.
		\end{enumerate}
	\end{example}
	
	\begin{example}In this example, we construct a compact weighted composition operator on $\Lmuinf$ for which neither induced multiplication or composition operators is compact on $\Lmuinf$.  We provide examples for which $\mu$ is both typical and atypical.
		
		\begin{enumerate}
			\item Define \[\mu(v) = \begin{cases}1/\modu{v} & \text{if $v \neq o$}\\1 & \text{if $v = o$},\end{cases}\] and 
			\[\psi(v) = \begin{cases}1 & \text{if $\modu{v}$ is odd}\\1/(\modu{v}+1) & \text{if $\modu{v}$ is even}.\end{cases}\]
			Let $(w_n)$ be a sequence in $T$ for which $\modu{w_n} = \lfloor \sqrt{n}\rfloor$ for all $n \in \N$ and define
			\[\varphi(v) = \begin{cases}w_{\modu{v}} & \text{if $\modu{v}$ is odd}\\v & \text{if $\modu{v}$ is even}.\end{cases}\]
			Observe $M_\psi$ and $C_\varphi$ are both bounded on $\Lmuinf$, but neither is compact.  However, $W_{\psi,\varphi}$ is compact on $\Lmuinf$ since
			\[\sup_{\modu{\varphi(v)}\geq N}\;\frac{\mu(v)}{\mu(\varphi(v))}\modu{\psi(v)} \leq \frac{1}{N}.\]
			
			\item Define \[\mu(v) = \begin{cases}\modu{v} & \text{if $v \neq o$}\\1 & \text{if $v = o$},\end{cases}\] and 
			\[\psi(v) = \begin{cases}1 & \text{if $\modu{v}$ is odd}\\1/(\modu{v}+1) & \text{if $\modu{v}$ is even}.\end{cases}\]
			Let $(w_n)$ be a sequence in $T$ for which $\modu{w_n} = n(n+1)$ for all $n \in \N$ and define
			\[\varphi(v) = \begin{cases}w_{\modu{v}} & \text{if $\modu{v}$ is odd}\\v & \text{if $\modu{v}$ is even}.\end{cases}\]
			Observe $M_\psi$ and $C_\varphi$ are both bounded on $\Lmuinf$, but neither is compact.  However, $W_{\psi,\varphi}$ is compact on $\Lmuinf$ since
			\[\sup_{\modu{\varphi(v)}\geq N}\;\frac{\mu(v)}{\mu(\varphi(v))}\modu{\psi(v)} = \frac{1}{N+1}.\]
		\end{enumerate}
	\end{example}
	
	In the next example, we see that composition operators on $\Lmuinf$ induced by bijections do not necessarily have inverses that are bounded in $\Lmuinf$.
	
	\begin{example} 
		In this example, we take $T$ to be $\Zed$ with root 0.  Define the weight $\mu$ and $\varphi$ on $T$ by \[\mu(v) = \begin{cases}
			1 & \text{if $v \geq 0$}\\
			|v| & \text{if $v < 0$}
		\end{cases}\] and
		\[\varphi(v) = \begin{cases}
			0 & \text{if $v = 0$}\\
			(v+1)/2 & \text{if $v > 0$ and odd}\\
			-v & \text{if $v > 0$ and even}\\
			2v+1 & \text{if $v < 0$}.
		\end{cases}\] By direct calculation,
		\[\frac{\mu(v)}{\mu(\varphi(v))} = \begin{cases}
			1 & \text{if $v = 0$, or $v > 0$ and odd}\\
			1/|v| & \text{if $v > 0$ and even}\\
			|v|/(2|v|+1) & \text{if $v < 0$}.
		\end{cases}\] As $\frac{\mu(v)}{\mu(\varphi(v))} \leq 1$ for all $v \in T$, $C_\varphi$ is bounded.  Note $\varphi$ is a bijection, and thus $C_{\varphi^{-1}}$ is well defined as an operator on $\Lmuinf$.  However, we see that $C_{\varphi^{-1}}$ is not bounded on $\Lmuinf$ because, if it were the case, then $\sup_{v \in T} \frac{\mu(v)}{\mu(\varphi^{-1}(v))}$ would be finite, which is equivalent to $\sup_{v \in T} \frac{\mu(\varphi(v))}{\mu(v)}$ being finite.  However, this is not true since for the sequence $v_n = 2n$ in $T$, we have  \[\lim_{n \to \infty} \frac{\mu(\varphi(v_n))}{\mu(v_n)} = \lim_{n \to \infty} |v_n| = \lim_{n \to \infty} 2n = \infty.\]
	\end{example}
	
	Fredholm composition and weighted composition operators acting on classical spaces of analytic functions over the unit disk typically arise from automorphic symbols; in other words, in most of the cases where Fredholm composition operators have been characterized, they are in fact invertible (see \cite{MacCluer:1997} and \cite{LoLoh:2019}).  Here we give a simple example to show this is not the case for our spaces. The example also illustrates Proposition \ref{proposition:SufficientforFredholmComposition}.
	
	\begin{example}
		Let $\mu$ be a weight on $T$.  For a fixed $w \in T^*$ define \[\varphi(v) = \begin{cases}v & \text{if $v \neq o$}\\w & \text{if $v = o$}.\end{cases}\]  Also take $\eta(v)=v$ for all $v \in T$.  It follows that both $C_\varphi$ and $C_\eta$ are bounded on $\Lmuinf$ and $C_\varphi C_\eta-I=C_\eta C_\varphi-I=C_\varphi-I.$  Moreover, $((C_\varphi-I)f)(v)=0$ if $v\neq o$ and $((C_\varphi-I)f)(o)=f(v_1)-f(o)$, which means $\textup{Im}(C_\varphi-I)=\{c\bigchi_o(v):c\in \Co\}$.  Thus $C_\varphi-I$ has finite rank and is compact.  Hence $C_\varphi$ is Fredholm.
	\end{example}
	
	In the final example, we construct a surjective isometric weighted composition operator (and thus Fredholm) on $\Lmuinf$ whose composition component is not bounded.
	
	\begin{example}
		For this example, we take $T$ to be $\Zed[i]$, the points in $\Co$ with integer real and imaginary parts, with root $o=0$.  In $T$ we define the quadrants as follows:
		\begin{align*}
			\mathrm{I} &= \left\{re^{i\theta} \in \Zed[i] : r > 0,\;0 \leq \theta < \pi/2\right\},\\
			\mathrm{II} &= \left\{re^{i\theta} \in \Zed[i] : r > 0,\;\pi/2 \leq \theta < \pi\right\},\\
			\mathrm{III} &= \left\{re^{i\theta} \in \Zed[i] : r > 0,\;\pi \leq \theta < 3\pi/2\right\},\\
			\mathrm{IV} &= \left\{re^{i\theta} \in \Zed[i] : r > 0,\;3\pi/2 \leq \theta < 2\pi\right\}.
		\end{align*}  Thus, $T = \{0\} \cup \mathrm{I} \cup \mathrm{II} \cup \mathrm{III} \cup \mathrm{IV}$.  On $T$, define the weight $\mu$ by \[\mu(v) = \begin{cases}1 & \text{if $v \in \mathrm{I}\cup\{0\}$}\\\modu{v} & \text{if $v \in \mathrm{II}$ or $\mathrm{IV}$}\\\modu{v}^2 & \text{if $v \in \mathrm{III}$}.\end{cases}\]   Define $\varphi:T \to T$ to be rotation by $\pi/2$, i.e. $\varphi(v) = e^{i\pi/2}v$ for all $v \in T$.  Thus $\varphi$ is a bijection with the root as the only fixed point.  We see that $C_\varphi$ is not bounded on $\Lmuinf$ since for any sequence $(v_n)$ in $\mathrm{III}$ with $\modu{v_n} \to \infty$ as $n \to \infty$, we have \[\lim_{n \to \infty}\frac{\mu(v_n)}{\mu(\varphi(v_n))} = \lim_{\modu{v} \to \infty}\frac{\modu{v}^2}{\modu{v}} = \infty.\]
		Finally, define $\psi:T \to \Co$ by \[\psi(v) = \begin{cases}\modu{v} & \text{if $v \in \mathrm{I}$ or $\mathrm{II}$}\\1/\modu{v} & \text{if $v \in \mathrm{III}$ or $\mathrm{IV}$}\\1 & \text{if $v = 0$}.\end{cases}\] First, observe $W_{\psi,\varphi}$ is bounded on $\Lmuinf$ since  
		\[\frac{\mu(v)}{\mu(\varphi(v))}\modu{\psi(v)} = 1\] for all $v \in T$.  By Theorem \ref{Theorem:WCOIso}, $W_{\psi,\varphi}$ is a surjective isometry.  So $W_{\psi,\varphi}$ is invertible, and thus Fredholm.
	\end{example}
	
	\section*{Acknowledgements}
	The authors would like to thank Ruben Mart\'inez-Avenda\~no of the Instituto Tecnol\'ogico Aut\'onomo de M\'exico for the idea to generalize the results to unbounded, locally finite metric spaces.  
	
	\bibliographystyle{amsplain} 
	\bibliography{references}
\end{document}